\documentclass[11pt]{amsart}

\usepackage{epsfig}
\usepackage{amsmath}
\usepackage{amssymb}
\usepackage{graphicx}

\usepackage{hyperref}
\usepackage{a4wide}

\theoremstyle{plain}
\newtheorem{prp}{Proposition}[section]
\newtheorem{lemma}[prp]{Lemma}
\newtheorem{stat}[prp]{Claim}
\newtheorem{corollary}[prp]{Corollary}
\newtheorem{theorem}[prp]{Theorem}
\theoremstyle{definition}

\newtheorem{prob}[prp]{Problem}

\usepackage{indentfirst,latexsym,bm}
\newcommand{\cg}{\mathcal{CG}}
\newcommand\GamL{{\sf \Gamma L}}
 \newcommand\soc{{\sf Soc}}
\newcommand\aut{{\sf Aut}} 
\newcommand\out{{\sf Out}} 
\newcommand\sym{{\sf Sym}}
\newcommand\alt{{\sf Alt}}

 \def\AGmL{\operatorname{A\Gamma L}}

\renewcommand{\leq}{\leqslant}

\newcommand{\tG}{\widetilde{G}}
\newcommand{\fix}{\mathrm{Fix}}
\newcommand{\fixs}{\mathrm{Fix}_{\Sigma}}

\newcommand{\cA}{\mathrm{A}}

\newcommand{\cV}{\mathrm{V}}

\newcommand{\cF}{\mathcal{F}}

\newcommand{\syl}{\mathrm{Syl}}

\newcommand{\rk}{\mathrm{rk}}

\usepackage[left=2cm,right=2cm,
    top=2cm,bottom=2.5cm,bindingoffset=0cm]{geometry}

\begin{document}

 \title[On vertex-transitive distance-regular covers of complete graphs...]
{On vertex-transitive distance-regular covers of complete graphs with an extremal smallest eigenvalue}

\author[L.~Yu.~Tsiovkina]{L.Yu. Tsiovkina}
\address{N.N.~Krasovskii Institute of Mathematics and Mechanics of UB RAS, 16 S.~Kovalevskaya str., Yekaterinburg, 620990 Russia,\\
Ural Federal University, 19 Mira Str., Yekaterinburg, 620002 Russia}
 \email{l.tsiovkina@gmail.com}
 
\begin{abstract} 
The paper is devoted to the study of abelian (in the sense defined by Godsil and Hensel) distance-regular $r$-covers of the complete graphs $K_n$. According to the construction by Coutinho, Godsil, Shirazi, and Zhan (2016), each such cover yields an equiangular set of lines of size $n$ that attains the relative bound. Moreover, there are four families of abelian covers that, through this construction, yield sets of lines attaining the absolute bound. All known representatives of these families --- the hexagon, the icosahedron graph, Taylor extensions of the Schl\"afli and McLaughlin graphs together with their distance-$2$ graphs, and three other examples arising from generalized quadrangles --- have a vertex-transitive automorphism group with at most two orbits on the arc set of the cover.
We aim to classify the covers from these families under the condition that the automorphism group of the cover is vertex-transitive and has at most two orbits on its arc set, which holds precisely when this group induces a transitive permutation group of rank at most $3$ on the set of cover fibres.
We apply several fundamental classification results on permutation groups of rank at most $3$ to describe the family of covers for which $r$ is odd and the smallest eigenvalue is extremal, equaling $-\sqrt{\sqrt{n}+1}$, which corresponds to lines in a complex Hilbert space. The results obtained encompass cases where the automorphism group of the cover induces a primitive or imprimitive group of rank at most $3$ on the set of fibres.
 
\end{abstract}

\maketitle
\vspace{2mm}

 \hspace{-17pt}{\it Keywords:}  distance-regular graph, antipodal cover, vertex-transitive graph, rank 3 group,   equiangular set of lines, equiangular tight frame.

 \hspace{-17pt}{\it 2020 MSC:} 05E18; 05E30

\section{Introduction} \label{Introduction}

 The paper is devoted to the study of
distance-regular antipodal covers of complete graphs. In view of \cite[Lemma 3.1]{GodsilHensel}, each such cover is equivalently defined as a connected graph whose vertex set admits a partition into a set of $n$ blocks of the same size $r\ge 2$ such that each block induces an $r$-coclique, the union of any two distinct blocks induces a perfect matching, and each two non-adjacent vertices lying in different blocks have exactly $\mu\ge 1$ common neighbours. The triple $(n,r,\mu)$ is called the cover parameters. Every cover with such parameters is briefly referred to as an $(n,r,\mu)$-cover. For an $(n,r,\mu)$-cover $\Gamma$, its blocks are called {\it fibres} and the set of all fibres is denoted by $\cF(\Gamma)$.
The group of all automorphisms of an $(n,r,\mu)$-cover $\Gamma$ that fix each of its fibres as a set is called the {\it covering group} of $\Gamma$ and is denoted by ${\mathcal{CG}}(\Gamma)$. If ${\mathcal{CG}}(\Gamma)$ is abelian and acts regularly on each fibre of $\Gamma$, then  $\Gamma$ is said to be  \emph{abelian} (see \cite{GodsilHensel}).

In this paper, we focus on  abelian $(n,r,\mu)$-covers.
The interest in this class of covers is due, among other things, to the fact that
they are one of the main known potential sources of new infinite families of non-trivial (i.e., distinct from orthonormal bases and regular simplices) equiangular tight frames (see the survey and references in \cite {FJMPW}).

Coutinho, Godsil, Shirazi, and Zhan \cite{CGSZ} established some significant connections between abelian $(n,r,\mu)$-covers and  equiangular lines in certain spaces. Specifically, using an old result due to Godsil and Hensel \cite{GodsilHensel}, they shown  that from every abelian cover one can derive a set of $n$ complex equiangular lines in spaces of dimensions $n - \frac{m_{\theta}}{r-1}$ and $n - \frac{m_{\tau}}{r-1}$, where $m_{\theta}$ and $m_{\tau}$ are the multiplicities of the second largest and the smallest eigenvalues $\theta$ and $\tau$ of the cover, respectively. The resulting sets of lines meet the equality in the relative bound, thus providing equiangular tight frames of size $n$ in the corresponding spaces.
The absolute bound for the number of equiangular lines imposes the following restrictions:
$$-\frac{1}{2}\sqrt{(n-1)(\sqrt{8n+1}-3)} \leq \tau \leq -\sqrt{\frac{1}{2}(\sqrt{8n+1}+3)}$$ for covers with even $r$, and
$$-(\sqrt{n}-1)\sqrt{\sqrt{n}+1} \leq \tau \leq -\sqrt{\sqrt{n}+1}$$ for covers with odd $r$.
The extreme cases in the upper or lower bounds imply the following parametrization of abelian covers, which, via this construction, yield sets of equiangular lines of maximum possible size in real or complex Hilbert spaces.

\begin{theorem}[{see \cite{CGSZ}}]\label{CGSZthm1}
Let $\Gamma$ be an abelian $(n,r,\mu)$-cover with eigenvalues $n-1>\theta>-1>\tau$.
\begin{itemize}
\item[$(A)$]
If $r$ is even, then $\Gamma$ gives a set of equiangular lines of size $n=d(d+1)/2$ in $\mathbb{R}^d$ via the Coutinho--Godsil--Shirazi--Zhan construction if and only if
either $(n,r,\mu)=(28,4,8)$ or   $r\ge 4$, $(i)$ $t=-\tau\in \mathbb{N}$, $(ii)$ $t\ge 3$ and is not divisible by four, $(iii)$ $\mu\ge 2$,
$(iv)$ $r$ divides $t-1$ whenever $2r\le t^2+1$, $(v)$  $\mu$ is even whenever $t$ is odd, $(vi)$ any odd prime divisor of $r$ is a divisor of $t-1$, and
\begin{equation} \label{params*reven1}(n,r,\mu)=((t^2-2)(t^2-1)/2,r,(t-1)^3(t+2)/(2r)),
\end{equation}
 or $r=2$,  $t=-\tau\in \{\sqrt{5}\}\cup \mathbb{N}$ and \begin{equation} \label{params*reven2}(n,r,\mu)\in\{((t^2-2)(t^2-1)/2,r,(t\pm1)^3(t\mp 2)/(2r))\}.\end{equation}

\item[$(B)$] If $r$ is odd, then $\Gamma$ gives a set of equiangular lines of size $n=d^2$ in $\mathbb{C}^d$ via the Coutinho--Godsil--Shirazi--Zhan construction if and only if
either $r=\mu=3=\sqrt{n}$ or $t=-\tau\in \mathbb{N}$, $(3,r)=1$, $r$ divides $t-1$, and
\begin{equation}\label{params*}
(n,r,\mu)=((t^2-1)^2,r,(t-1)^2(t^2+t-1)/r).
\end{equation}
\end{itemize}
\end{theorem}

The existence of abelian covers with even $r$ satisfying conditions of Theorem \ref{CGSZthm1}$(A)$ is known only for the parameter sets $(t,r)= (2, 2)$, $(\sqrt{5}, 2)$, $(3, 2)$ and $(5,2)$, and $(n,r,\mu)=(28,4,8)$. These correspond to the hexagon, the icosahedron graph, Taylor extensions of the Schl\"{a}fli and McLaughlin graphs (as well as their distance-2 graphs), and the collinearity graph of a generalized quadrangle $GQ(3,9)$ with a deleted spread. All these examples are arc-transitive and yield all known examples of real equiangular sets of lines that attain the absolute bound (in dimensions $d\in\{2, 3, 7, 23\}$).

For the parameter set $(n,r,\mu) = (9, 3, 3)$, there exist exactly two non-isomorphic (abelian) covers (see \cite[p. 386]{BCN}). In both examples, the full automorphism group is vertex-transitive and has at most two orbits in its induced action on the arc set. The question of the existence of abelian covers with parameters \eqref{params*} remains open. It is of particular interest, as  covers with odd $r$ satisfying conditions of Theorem \ref{CGSZthm1}$(B)$ give rise to equiangular tight frames that are equivalent to so-called SIC-POVM sets, which play an important role in some problems of quantum information theory (see the discussion in \cite{CGSZ}).

This raises the following question.

\begin{prob}
Classify $G$-vertex-transitive abelian covers $\Gamma$ with parameters \eqref{params*reven1}, \eqref{params*reven2} or \eqref{params*} with a small number of $G$-orbits in the induced action of $G$ on the arc set of $\Gamma$.

\end{prob}

The author's paper \cite{Tsi24} completed the classification of arc-transitive covers with non-solvable automorphism groups.

\begin{theorem}[see \cite{Tsi24}] Suppose $\Gamma$ is an unknown edge-transitive $(n,r,\mu)$\--cover and $\Sigma=\cF(\Gamma)$.
Then $\aut(\Gamma)^{\Sigma}\le \AGmL_1(|\Sigma|)$ and $|\cg(\Gamma)|=r$. If, in addition, $\aut(\Gamma)$ acts transitively on arcs,
then it is solvable, $\mu>1$ and $n=r\mu$.
\end{theorem}

We formulate below  several results on abelian covers, which follow from the classification obtained.

\begin{theorem}[see \cite{Tsi24}]
Arc-transitive abelian $(n,r,\mu)$-covers with nonsolvable automorphism groups are exhausted by
\begin{itemize}\item distance-transitive Taylor covers, \item Thas--Somma covers,
\item collinearity graphs of  generalized quadrangles dual to ${\rm H}(3,q^2)$'s with a deleted cyclic spread,
\end{itemize}
together with their quotient covers obtained by quotiening by orbits of proper subgroups of the covering group.

All examples of edge-transitive abelian  $(n,r,\mu)$-covers $\Gamma$ with ${\aut}(\Gamma)^{\cF(\Gamma)}\not\le \AGmL_1(|\cF(\Gamma)|)$ are known.
\end{theorem}
\begin{corollary}
The only edge-transitive abelian $(n,r,\mu)$-covers with parameters \eqref{params*reven1} or  \eqref{params*reven2} are the hexagon, the  icosahedron graph, Taylor extensions of the Schläfli and McLaughlin  graphs, and  their distance-$2$ graphs.
\end{corollary}
\begin{corollary}
The only edge-transitive abelian $(28,4,8)$-cover is the collinearity graph of the generalized quadrangle dual to $H(3,9)$ with a  deleted cyclic spread.
 \end{corollary}
\begin{corollary}[see \cite{Tsi24}]
There is no edge-transitive abelian $(n,r,\mu)$-cover with parameters \eqref{params*}.
\end{corollary}

The main goal of this paper is to classify abelian covers $\Gamma$ with parameters \eqref{params*} whose full automorphism group is vertex-transitive and has exactly two orbits in its induced action on the arc set of $\Gamma$. This condition is satisfied precisely when the group $G=\aut(\Gamma)$ induces a transitive permutation group of rank $3$ on the set $\Sigma$ of cover fibres (see Lemma \ref{Prel:Lemma1} in Section \ref{Preliminaries}).

We base our arguments on several fundamental classification results on permutation groups of rank  at most $3$, including the classification of finite $2$-transitive permutation groups. In what follows, we consider separately the cases of primitive and imprimitive groups $G^{\Sigma}$.
\medskip

The primitive permutation groups $X$ of degree $n$ and rank $3$ are classified (see the discussion and references in \cite[Chapter 11]{BM}) and are divided into three types:
\begin{itemize}
\item[$(1)$] { Almost simple, that is, $S\trianglelefteq X\le \aut(S)$, where $S$ is a non-abelian simple group};
 \item[$(2)$] Affine, that is, $X$ has a normal regular elementary abelian group of order $n$, so $n$ is a prime power;
\item[$(3)$] Wreathed type, that is, $T \times T \trianglelefteq X \leq T_0 \wr S_2$, where $T_0$ is a $2$-transitive group of degree $n_0$ with simple socle $T$ and $n=n_0^2$.
\end{itemize}

By one of the above-mentioned results   (e.g. see \cite[p. 386]{BCN}), there are precisely two abelian covers
satisfying conditions   of Theorem \ref{CGSZthm1}$(B)$, which  have affine  groups $G^{\Sigma}$ (of rank at most $3$). Abelian covers with general parameters having almost simple  primitive groups $G^{\Sigma}$ were studied in a series of papers by the author (see references in \cite{Tsi24}). %\cite{Tsi21,Tsi22a, Tsi22b}.
Here we prove the following result, describing abelian covers with primitive groups $G^{\Sigma}$ of the third type.

\begin{theorem}\label{MainThm:PrimitiveCase} Let $\Gamma$ be an abelian cover with parameters \eqref{params*}, $G=\aut(\Gamma)$, and $\Sigma=\cF(\Gamma)$. If $G^{\Sigma}$ is a primitive group of rank $3$, then it is either almost simple or has wreathed type, $T\times T \trianglelefteq G^{\Sigma} \le T_0 \wr 2$, where $T_0$ is a $2$-transitive group of degree $n_0 = t^2 - 1=\frac{q^4-1}{q-1}$ such that $T=\soc(T_0)\simeq L_4(q)$.
\end{theorem}

Imprimitive permutation groups $X$ of rank $3$ have a common characterization involving several $2$-transitive actions. Namely, every such group $X$ has a (unique) nontrivial imprimitivity system $\mathcal{B}$ and by the ``Embedding Theorem'' for imprimitive groups, the group   $X$ can be viewed as a subgroup of the wreathed product $H\wr Y$, where, for a block $B\in \mathcal{B}$, $H={X_{\{B\}}}^B$,  and the groups $H$   and $Y\simeq G^{\mathcal{B}}$ are both $2$-transitive. By \cite{DGLPP}, provided that $H$ is an almost simple group, the group $X$, except for a few sporadic subcases, is either quasi-primitive (and known) or full, i.e. the kernel of the action of $X$ on $\mathcal{B}$ contains the $|\mathcal{B}|$-th direct power of the socle of the wreathed component $H$.
(For the case of an affine group $H$, the structure of  groups $X$ may be very complicated and the problem of their classification is far from its solution, see a discussion in \cite{DGLPP}). Using this result together with the classification of $2$-transitive permutation groups, we show the next theorem, which reduces the consideration of the imprimitive case for $G^{\Sigma}$ to    some rather special situations.

\begin{theorem}\label{MainThm:ImprimitiveCase} Suppose that $\Gamma$ is an abelian cover with parameters \eqref{params*}, $\Sigma=\cF(\Gamma)$, $G=\aut(\Gamma)$, and $G^{\Sigma}$ is an imprimitive group of rank $3$ with subdegrees $1$, $k_1$, and $k_2$.
Let $F\in \Sigma$, $\mathcal{B}$ be the unique nontrivial imprimitivity system for $G^{\Sigma}$, $B$ be its block (of size $k_1+1$) containing $F$,
$M=G_{\{F\}}$, $K=\cg(\Gamma)$ and $H={{G_{\{B\}}}}^B$. Then $H$ is an affine group (of degree $k_1+1=p^l$, where $p$ is a prime) and one of the following holds:
\begin{itemize}
\item[$(1)$] $G^{\mathcal{B}}$ is an almost simple group and either
\begin{itemize}
\item[$(i)$] $(\soc(G^{\mathcal{B}}), |\mathcal{B}|) = (L_2(8), 9)$ and $l$ is even, or
\item[$(ii)$] $(\soc(G^{\mathcal{B}}),|\mathcal{B}|)=(L_5(3),121)$,
$t=12$, $r=11$, $p=13$, $l=2$, and $M^B\le GL_2(13)$ or $M^B\le \Gamma L_1(13^2)$;
\end{itemize}
\item[$(2)$]
$G^{\mathcal{B}}$ is an affine group, $t$ is even, $t-1=p_1^{s_1}=p_2^{s_2}-2$ for some  primes $p_1$ and $p_2$, and $\{|B|,|\mathcal{B}|\}=\{(t-1)^2,(t+1)^2\}$.
\end{itemize}
\end{theorem}

This paper is organized as follows.
Section~\ref{Preliminaries} provides some necessary definitions and auxiliary results.
In Sections~\ref{PrimitiveCase} and \ref{ImprimitiveCase} we prove Theorem~\ref{MainThm:PrimitiveCase} and Theorem~\ref{MainThm:ImprimitiveCase} respectively.

\section{Preliminaries}
\label{Preliminaries}

Throughout this paper, we consider only undirected graphs without loops or multiple edges. For a graph $\Gamma$, we denote the set of its vertices by $\cV(\Gamma)$ and the set of its arcs (i.e., ordered pairs of adjacent vertices) by $\cA(\Gamma)$.

Let $\Gamma$ be an $(n,r,\mu)$-cover. For a vertex $a \in \cV(\Gamma)$, we denote by $\Gamma_i(a) $ the sphere of radius $ i $ centered at $a$ in $\Gamma$ (we also use the notation $[a] = \Gamma_1(a)$ when $\Gamma$ is clear from the context), and by $F(a)$ we denote the fibre containing $a$.
Every two adjacent vertices in $\Gamma$ have a constant number $\lambda$ of common neighbors, which equals $n - (r-1)\mu - 2$. Then, $\Gamma$ has exactly four distinct eigenvalues: $ n-1 > \theta > -1 > \tau$ (see \cite{BCN}), with multiplicities $1, m_{\theta}, n-1, m_{\tau}$, respectively. Here, $ \theta $ and $ \tau $ are the roots of the equation $ x^2 - (\lambda - \mu)x - (n-1) = 0$, and the multiplicities $ m_{\theta} $ and $ m_{\tau} $ can be computed using the formulas:
 $$m_{\theta}=\frac{-\tau(r-1)n}{\tau+\theta}, \quad m_{\tau}=\frac{\theta(r-1)n}{\tau+\theta}.$$

It can be easily seen that $\Gamma$ is antipodal, with diameter   $3$, and the antipodal classes coincide with the cover fibres. According to \cite{GodsilHensel}, quotiening of $\Gamma$ by the orbits of a subgroup of $\mathcal{CG}(\Gamma)$ may yield a cover with new parameters $(n, r', \mu')$. More specifically, for any non-trivial subgroup $N$ of $\mathcal{CG}(\Gamma)$ of order less than $r$, the \emph{quotient cover} $\Gamma^N$, defined as a graph on the set of $N$-orbits where two orbits are adjacent if there is an edge in  $\Gamma$ between them, is a $(n, r/|N|, \mu |N|)$-cover.

Let $G$ be a finite permutation group acting on a set $\Omega$, and let $\Lambda \subseteq \Omega$. The sets $G_{\{\Lambda\}} = \{ g \in G \mid \Lambda^g = \Lambda \}$ and $G_{\Lambda} = \{ g \in G \mid x^g = x \text{ for all } x \in \Lambda \}$ form two subgroups of $G$, called as the \emph{setwise stabilizer} and the \emph{pointwise stabilizer} of $\Lambda$ in $G$, respectively.
If $g \in G_{\{\Lambda\}}$, then $g$ is said to fix $\Lambda$ as a set; in particular, if $g \in G_{\Lambda}$, then $g$ fixes $ \Lambda$ pointwise.
If $\Lambda$ is a $G$-invariant set (i.e., $G = G_{\{\Lambda\}}$), then $G^{\Lambda}$ denotes the permutation group induced by $G$ on $\Lambda$ (so that $G^{\Lambda}\simeq G/G_{\Lambda}$).
The set of all points in $\Omega$ fixed by an element $g \in G$ is denoted by $\fix_{\Omega}(g)$  or simply $\fix(g)$ when it is clear which $\Omega$ is meant. A \emph{$G$-orbital} of $G$ is any orbit in the induced action of $G$ on  $\Omega^2$. If $G$ is transitive, then the number of $G_a$-orbits on $\Omega$ is referred to as the \emph{(permutation) rank} of $G$ and is denoted by $\rk(G)$ (so $\rk(G)$ coincides with the number of $G$-orbitals).
Further, for a natural number $l$, $\pi(l)$ denotes the set of its prime divisors, and for a finite group $X$, $\pi(X)$ refers to its \emph{prime spectrum}, i.e., the set $\pi(|X|)$. For any $p \in \pi(l)$, $(l)_p$ denotes the $p$-part of $l$, which is the largest power of $p$ dividing $l$, while $(l)_{p'}$ denotes the number $l / (l)_p$, called  as the $p'$-part of $l$. Finally,
$(x,y)$ denotes the greatest common divisor of the integers $x$ and $y$.

The remaining general terminology and notation used in this paper are mostly standard and can be found in \cite{Asch,DM, BCN}.

\begin{lemma}\label{Prel:Lemma1} If $\Gamma$ is an $(n, r, \mu)$-cover, $|\cg(\Gamma)|=r$, and $X$ is its transitive group of  automorphisms, then the group $G:= \cg(\Gamma)X$ has exactly $\rk(G^{\Sigma})-1$ orbits on the set of arcs of the cover $\Gamma$, where $\Sigma=\cF(\Gamma)$.
\end{lemma}
\begin{proof}
Let $\Delta_1,\ldots,\Delta_s$ be the complete set of $G$-orbitals whose union gives   $\cA(\Gamma)$. Then for each vertex $a$ there are exactly $s$ arcs-representatives of the form $(a,a_i)\in \Delta_i$. All vertices $a_i$ belong to pairwise distinct fibres. Moreover, if $F(a_i)^g=F(a_j)$ for some $1\le i,j\le s$ and $g\in G_{\{F(a)\}}$, then $\cg(\Gamma)$ contains an element $h$ such that $(a^{gh},(a_i)^{gh})=(a,a_j)$, hence $i=j$. Thus $s\le \rk(G^{\Sigma})-1$. On the other hand, in view of the action of $\cg(\Gamma)$ on $\cV(\Gamma)$, there are at most $s$ non-diagonal $G^{\Sigma}$-orbitals on $\Sigma$, which yields the required equality.
\end{proof}
\medskip

Recall that, by a theorem of Burnside, every finite $2$-transitive group is either almost simple or affine. Their classification (see Theorem \ref{2tra} below), which we give in accordance with \cite[Theorem 2.9]{GLP} (see   also \cite[Theorem 11.2.1]{BM})  will provide a basis for our proofs.

\begin{theorem}\label{2tra} Let $G$ be a finite $2$-transitive
permutation group on a set $X$, $|X|=n$, $a\in X$, $H=G_a$, and $T$ be the socle of $G$.

Then either

\begin{itemize}

\item[$(1)$] $G$ is an almost simple group and for $(T,n)$ one of the following
possibilities holds:
\begin{itemize}

\item[$(i)$] alternating $(\alt_n,n)$, $n\ge 5$;

\item[$(ii)$] linear $(L_m(q),(q^m-1)/(q-1))$, $m\ge 2$ and $(m,q)\notin \{(2,2),(2,3)\}$;

\item[$(iii)$] symplectic $(Sp_{2m}(2),2^{2m-1}\pm 2^{m-1})$, $m\ge 3$;

\item[$(iv)$] unitary $(U_3(q),q^3+1)$, $q\ge 3$;

\item[$(v)$] Ree $(^2G_2(q),q^3+1)$, $q=3^{2a+1}\ge 27$;

\item[$(vi)$] Suzuki $(Sz(q),q^2+1)$, $q=2^{2a+1}\ge 8$;

\item[$(vii)$] Mathieu $(M_n,n)$, $n\in \{11,12,22,23,24\}$;

\item[$(viii)$] sporadic $(L_2(11),11)$, $(M_{11},12)$, $(\alt_7,15)$, $(L_2(8),28)$,
$(HiS,176)$, $(Co_3,276)$,
\end{itemize}
or

\item[$(2)$] $G=TH$, $T$ is an elementary abelian group of order $n=p^m$, where $p$ is a prime,
and one of the following holds:
\begin{itemize}

\item[$(i)$] linear $m=cd$, $d\ge 2$ and $SL_d(p^c)\trianglelefteq H\le \Gamma L_d(p^c)$;

\item[$(ii)$] symplectic $m=cd$, $d$ is even, $d\ge 4$ and
$Sp_d(p^c)\trianglelefteq H\le Z_{p^c-1}\circ \Gamma Sp_d(p^c)$;

\item[$(iii)$] $G_2$-type $m=6c$, $p=2$ and
$G_2(2^c)'\trianglelefteq H\le Z_{2^c-1}\circ {\rm Aut}(G_2(2^c))$;

\item[$(iv)$] one-dimensional $H\le \Gamma L_1(p^m)$;

\item[$(v)$] exceptional
$p^m\in \{9^2,11^2,19^2,29^2,59^2\}$ and $SL_2(5)\trianglelefteq H$, or
$p^m=2^4$ and $\alt_6$ or $\alt_7\trianglelefteq H$, or
$p^m=3^6$ and $SL_2(13)\trianglelefteq H$;

\item[$(vi)$] extraspecial $p^m\in \{5^2,7^2,11^2,23^2\}$ and
$SL_2(3)\trianglelefteq H$ or $p^m=3^4$, $R=D_8\circ Q_8\triangleleft H$,
$H/R\le \sym_5$ and $5$ divides $|H|$.\end{itemize}
\end{itemize}
\end{theorem}
Note that the socle $T$ of an almost simple $2$-transitive group of degree $n$ is also $2$-transitive, with the only exception of the case $(T,n)=(L_2(8),28)$, in which $T$ acts primitively with rank $4$ (e.g. see \cite[Lemma 2.10]{GLP}).
\medskip

The following   number-theoretic results will  be used in the forthcoming arguments.
\begin{lemma}\label{Prel:LemmaNT}
Let $q$ be a prime power, $e\in \{1, -1\}$, and $k,m\in \mathbb{N}$. Then
the following statements hold:
\begin{itemize}

\item[$(1)$] if an odd prime $r$ divides $q-e$, then $(q^m -e^m)_r = (m)_r (q -e)_r$ \emph{(see \cite[Lemma 6(ii)]{Zavarnitsine2004})};
 \item[$(2)$] if $4$ divides $q -e$ or $m$ is odd, then $(q^m - e^m)_2 = (m)_2(q- e)_2$
 \emph{(see \cite[Chapter IX, Lemma 8.1]{HuppertBlackburnPartII})}.
  \item[$(3)$] $(q^k-1, q^m-1) = q^{(k,m)}-1$ \emph{(e.g. see {\cite[Lemma 6(iii)]{Zavarnitsine2004}})};
\end{itemize}

\end{lemma}

\begin{lemma}[A Corollary of Zsigmondy's Theorem \cite{Zsigmondy}]\label{Prel:LemmaZsigmCor} If $p$ and $q$ are prime numbers, $1\le m,n\in \mathbb{N}$ and
$$p^m=q^n+1,$$ then one of the following assertions holds:
\begin{itemize}
\item[$(1)$] $q=2, p=3, n=3, m=2$;

\item[$(2)$] $q=2, m=1$, $n$ is a power of $2$ and $p=q^n+1$ is a Fermat prime;

\item[$(3)$] $p=2, n=1$ and $q=p^m-1$ is a Mersenne prime, in particular, $m$ is prime.
\end{itemize}
\end{lemma}

Throughout the rest of the paper we make and use the following assumptions and notations:
\begin{itemize}
\item $\Gamma$ is an abelian cover with parameters \eqref{params*}, that is,
\begin{equation*} %\label{params}
(n,r,\mu)=\big((t^2-1)^2,r,(t^2+t-1)(t-1)\frac{(t-1)}{r}\big),
\end{equation*} where $t$ is a positive integer and $r$ is a divisor of $t-1$ such that $(6,r)=1$, and eigenvalues
$$
k=t^2(t^2-2),\quad
\theta=(t^2-2)t,\quad
-1,\quad
\tau=-t
$$
of respective multiplicities
$$
1,\quad
m_{\theta}=(t^2-1)(r-1),\quad
k,\quad
m_{\tau}=(t^2-2)(t^2-1)(r-1),
$$
and
 \item $\Sigma=\cF(\Gamma)$ (so $n=|\Sigma|$), $v=nr$,
 $a\in F\in \Sigma$,
 \item $K=\cg(\Gamma)\le G\le \aut(\Gamma)$, $G$ is transitive on $\cV(\Gamma)$,
 $M=G_{\{F\}}$ and $C=G_F$,
 \item $\alpha_i(g)=|\{x\in \cV(\Gamma){:}\ \partial(x,x^g)=i\}|$ for
an element $g\in G$.
\end{itemize}
\medskip

Moreover, $$\mu=(t^3-2t+1)\frac{(t-1)}{r},$$
$$\lambda=(t^3-2t+1)\frac{(t-1)}{r}+t(t^2-3),$$ and in view of the Hoffman bounds for cliques and cocliques (see \cite[Propositions 4.4.6, 3.7.2]{BCN})
\begin{equation} \label{HofmanClique} |Q|\le 1+(t^2-2)t= \frac{r\mu}{t-1}
\end{equation}
for any clique $Q$
of the graph $\Gamma$, and
\begin{equation} \label{HofmanCoclique} |C|\le \frac{r(t^2-1)^2}{1+(t^2-2)t}=\frac{rn}{r\mu/(t-1)}=\frac{n(t-1)}{\mu}
\end{equation}
for any coclique $C$ of the graph $\Gamma$.

\smallskip

Next we list a number of auxiliary lemmas on  basic properties of the pair $(\Gamma,G)$.

\begin{lemma}\label{Prel:Lemma2}

The class of covers $\Gamma$ with parameters \eqref{params*}
is closed under taking the quotient covers $\Gamma^U$ obtained by factorizing by subgroups $1<U<K$.
\end{lemma}
\begin{proof}
It is enough to note that for any subgroup $1<U<K$ the graph $\Gamma^U$ is a $((t^2-1)^2,r',(t^3-2t+1)(t-1)/r'))$-cover, where $|K:U|=r'$, which follows from \cite{GodsilHensel}.
\end{proof}

\begin{lemma}[see {\cite[Lemma 1]{Tsi24}}]\label{Prel:Lemma3}The following statements are true:
 \begin{enumerate}

\item $C=C_G(K)\cap G_a$  and  $M=K:G_a$;
\item  $|G:M|=(t^2-1)^2$ and $|G:G_a|$ divides $(t^2-1)^2(t-1)$;
\item $\pi(G)=\pi(nr)\cup \pi(G_a)\subseteq \pi\big(n!\cdot(t-1)\big)$;

\item $|\fix(G_a)|=|N_G(G_a):G_a|$ divides $nr$, and $|\fixs(M)|=|N_G(M):M|$ divides $n$;
\item if $p\in \pi(G)$ and $p>t+1$, then $p\in \pi(G_F)$;
\item $G/C_G(K)\le \aut(K)$ and if $r$ is prime, then $G/C_G(K)\le Z_{r-1}$ and every element $g\in G$ of prime order $p\not\in \pi(r-1)$ either has no fixed points or $\fix(g)$ is the union of some fibres;
\item if $g$ is an element of prime order $p$ in $G$, $p>r$ and $p\not\in \pi((t^2-1)(t^2-2)t)$, then $r_3=r-1$, $r_2r_1>0$, and $|\Omega|\ge r(1+r_1)+r\cdot z$ for some non-negative integer $z$ which is a  multiple of $p$, where $r_i=|\Gamma_i(a)| \pmod p$, $i=1,2,3$.

 \end{enumerate}
\end{lemma}

\begin{lemma}[see {\cite[Lemma 2]{Tsi24}}]\label{Prel:Lemma4} Let $p\in \pi(G)$. If $p> t-1$ and $p$ does not divide $t+1$, then $$\max\{p,|\fixs(g)|\}< \mu;$$ in particular,
if $p>\mu/2$, then for any element $g\in G$ of order $p$, the subgraph $\fix(g)$ is an abelian $(|\fixs(g)|,r,\mu')$-cover with $\mu'\equiv \mu\pmod p$,  $\mu>|\fixs(g)|> (r-1)\mu'$ and $\pi(r)\subseteq \pi(|\fixs(g)|)$.

\end{lemma}

\begin{lemma}[see {\cite[Lemma 3]{Tsi24}}]\label{Prel:Lemma5} Suppose $g$ is an involution of   $G$ and $\Omega=\fix(g)\ne\varnothing$. Let $f=|\Omega\cap F(z)|$ for $z\in \Omega$, $l=|\fixs(g)|$, and $X=\{x\in \Gamma\setminus\Omega |\ [x]\cap \Omega\ne \varnothing\}$. Then $\Omega$ is a regular graph of degree $l-1$ on $lf$ vertices and the following statements hold:
\begin{itemize}
\item[$(1)$] if $l=1$, then $t$ is even, $\alpha_3(g)=0$, $\alpha_1(g)+\alpha_2(g)=(n-1)r$, and $\Omega$ coincides with the fibre $F(z)$;
\item[$(2)$] if $l>1$, then $\alpha_3(g)=(r-f)l$, $\alpha_1(g)+\alpha_2(g)=(n-l)r$, each vertex in $X$ is "on average" adjacent to $\frac{fl(n-l)}{|X|}$ vertices in $\Omega$, and the number of such vertices in $\Omega$ does not exceed $l$ (in particular, if $|X|=f(n-l)$, then $l\le \lambda$) and $$f\le \frac{|X|}{n-l}\le|\Omega|\le \frac{ (\lambda-\mu)\alpha_1(g)}{n-l} +r\mu\le r\lambda,$$ furthermore, if $F(z)\subset\Omega$, then $X=\Gamma\setminus\Omega$ and either $(i)$ $\alpha_1(g)=(n-l)r$, $l\le\lambda$, and each $\langle g\rangle$-orbit on $\Gamma\setminus\Omega$ is an edge, or $(ii)$ $\alpha_2(g)>0$ and $l\le \mu$;
\item[$(3)$] if $t$ is even, then $X=\{x\in \Gamma |\ \partial(x,x^g)\in \{1,2\}\}$;
\item[$(4)$] if $f=1$ and $l>1$, then $\Omega$ is an $l$-clique and $l\le r\mu/(t-1)\le \mu$;
\item[$(5)$] if $f>1$, $t$ is even, and $l>1$, then the diameter of the graph $\Omega$ is 3 and $|\Omega|\le l+(l-1)^2(l-2).$
\end{itemize}
\end{lemma}

It is easy to see that the groups ${G_a}^{[a]}$ and ${G_{\{F\}}}^{\Sigma-\{F\}}$ are permutation isomorphic. This implies the following useful fact.
\begin{lemma}[see {\cite[Lemma 4]{Tsi24}}]\label{Prel:Lemma6}Suppose that $G^{\Sigma}$ has permutation rank $3$ with  subdegrees $1\le k_1\le k_2$. Then
\begin{itemize}
 \item[$(i)$]
$G_a$ has exactly two orbits on $[a]$, say $X_1$ and $X_2$, with lengths $k_1$ and $k_2$ respectively, satisfying
 \begin{equation} \label{eq1}
k_1(\lambda-\lambda_1)=k_2(\lambda-\lambda_2),
\end{equation}
where $\lambda_i=|[x_i]\cap X_i|$ for a vertex $x_i\in X_i$, $i=1,2$,

\item[$(ii)$]  if  $G_a$ fixes a vertex $a^*\in F(a)-\{a\}$, then
\begin{equation} \label{mueq1*}
k_1(\mu-\mu_1)=k_2(\mu-\mu_2),
\end{equation}
where $\mu_1=|[y_1]\cap X_1|$ for a vertex $y_1\in X_1^*$,
$\mu_2=|[y_2]\cap X_2^*|$ for a vertex $y_2\in X_2$, and, for $i=1,2$, $X_i^*$ is a $G_a$-orbit on $[a^*]$ of length $k_i$.
\end{itemize}

\end{lemma}

% \begin{prp}[see {\cite[Proposition 1]{Tsi24}}] The rank of $G^{\Sigma}$ is not equal to $2$.
% \end{prp}
 \section{Proof of Theorem~\ref{MainThm:PrimitiveCase}}
\label{PrimitiveCase}

In this section we will assume that $G^{\Sigma}$ is a primitive group of rank $3$
with subdegrees $1$, $k_1$, and $k_2$, where $k_1\le k_2$. Under this assumption, by \cite[Theorem 2.6.14]{ChenPonomarenkoLNCC} and \cite{Berggren}, the group $G^{\Sigma}$ of degree $(t^2-1)^2$ has two self-paired orbitals on $\Sigma$.

Since $n=(t^2-1)^2$ is not a prime power, then by the classification of primitive groups of rank $3$ (see, e.g., \cite[Chapter 11, Theorem 11.1.1]{BM})
the group $G^{\Sigma}$ is either almost simple or is of wreathed type.

\medskip

Suppose that the group $G^{\Sigma}$ is of wreathed type, that is, $P=T \times T \trianglelefteq G^{\Sigma} \leq T_0 \wr S_2$, where $T_0$ is a $2$-transitive group of degree $n_0$ on a set $\Delta$ with a simple non-abelian socle $T$, and $n=n_0^2$. Recall that in this case the action of the wreath product $T_0\wr S_2$ on the set $\Sigma$ which is identified with the cartesian square $\Delta^2$ is given by the rule
$$(\delta_1,\delta_2)^{(g_1,g_2)\tau}=(\delta_1^{g_1},\delta_2^{g_2})^{\tau}=({\delta_{(1){\tau}}}^{g_{(1)\tau}},{\delta_{{(2)\tau}}}^{g_{(2)\tau}})$$
for all $(\delta_1,\delta_2)\in \Delta^2$ and $(g_1,g_2)\tau\in T_0\wr S_2$, and the nontrivial subdegrees of the group $G^{\Sigma}$ are equal to $k_1=2(n_0-1)$ and $k_2=(n_0-1)^2$. Simplifying the equation \eqref{eq1}, we get
\begin{equation}\label{eq1*}
2(\lambda-\lambda_1)=(n_0-1)(\lambda-\lambda_2).
\end{equation}

Note that $G_a$ contains a subgroup $A=S \times S$, where $S$ is the stabilizer of a point in $T$, which has exactly two orbits on $X_1$, say $Y_1$ and $Y_2$, so that $A^{Y_i} \simeq S$ and the lengths of these orbits are $|Y_1|=|Y_2|=n_0-1$.
So $\lambda_1$, being a sum of some subdegrees of $A^{X_1}$, must be a sum of some subdegrees of groups $A^{Y_1}$ and $A^{Y_2}$.

By \cite{Tsi24}, $T$ cannot be an alternating group of degree $n_0$. Also $T$ cannot be a sporadic simple group, since
in this case $n_0\in\{11,15,12,22,23,24,176,276\}$ and the condition $t^2-1=n_0$ gives $n_0=15,24$, which contradicts the condition $t\ge 6$.

In what follows, we will apply the identity \eqref{eq1*} to show that $\lambda_1$ cannot be a sum of subdegrees  of $A^{X_1}$. In doing so, we exclude almost all other admissible cases for $T$ described by Theorem \ref{2tra}.

1. Suppose that $n_0-1=q=p^e$, where $p$ is a prime. Since $t\ge 6$ and $t^2-2=n_0-1$, $p>2$, $e$ is odd, and $p^e$ divides $\lambda-\lambda_1$.
Thus $T\not\simeq Sz(q)$. Let us consider the cases when $T\simeq L_2(q),{^2G}_2(q) $ or $U_3(q)$.

First, note that since $\lambda_1+1\le k_1$ and $$\lambda=%(t^2-2)\big(\frac{(t-1)^2}{r}+\frac{t-1}{r}+t\big)+\frac{t-1}{r}-t=
(t^2-2)z+\frac{t-1}{r}-t,$$
where $z=\frac{(t-1)^2}{r}+\frac{t-1}{r}+t$,
then $$\lambda_1=z_1(t^2-2)-t+\frac{t-1}{r},$$
where $z_1\in\{ 1,2\}$, and $$\lambda_2=\lambda-2(z-z_1)=(t^2-4)z+\frac{t-1}{r}-t+2z_1.$$

The facts about subdegrees of point stabilizers in the corresponding $2$-transitive groups used below are well-known (see, for example, \cite{Wilson}, \cite{DM} and also \cite{Tsi22c}).

\smallskip

In the case $T \simeq L_2(q)$, for any $i\in\{1,2\}$, the stabilizer of a point in $A^{Y_i}$ has exactly three orbits on $Y_i$: one orbit of length 1 and two orbits of length $\frac{q-1}{2}$.
Hence $\lambda_1 \equiv 0\pmod {\frac{t^2-3}{2}}$ or $\lambda_1 \equiv 1\pmod {\frac{t^2-3}{2}}$. But then $z_1-t+{\frac{t-1}{r}}\equiv 0$ or $1\pmod {\frac{t^2-3}{2}}$, a contradiction.

In the case of $T \simeq {^2G}_2(q)$, for any $i\in\{1,2\}$, the stabilizer of a point in $A^{Y_i}$ contains a cyclic subgroup of order $(q-1)/2$, which has one orbit of length 1 on $Y_i$ and the length of any other its orbit   on $Y_i$ is $(q-1)/2$.
Hence $\lambda_1 \equiv 0\pmod {\frac{t^2-3}{2}}$ or $\lambda_1 \equiv 1\pmod {\frac{t^2-3}{2}}$. %But then $z_1+t-{\frac{t-1}{r}}\equiv 0$ or $1\pmod {\frac{t^2-3}{2}}$,
Contradiction, as above.

In the case $T \simeq U_3(q)$, for any $i\in\{1,2\}$, the stabilizer of a point in $A^{Y_i}$ has one orbit of length 1 on $Y_i$, and the length of any other its orbit on $Y_i$ is a multiple of $q-1$.
Hence $\lambda_1 \equiv 0\pmod {t^2-3}$ or $\lambda_1 \equiv 1\pmod {t^2-3}$,
again a contradiction.

\smallskip

2. For $T \simeq Sp_{2d}(2)$, $d\ge 3$
we have $n_0=2^{2d-1}\pm 2^{d-1}$.
\smallskip

2.1. Let $(t-1)=2x$ and $(t+1)=2^{d-2}y$, where $xy=2^{d}\pm 1$. Then $2^{d-3}$ divides $y\pm 1$,
which implies $d\le 6$. But for $3\le d\le 6$ only two cases are possible: $t=11$ and $d=4$, or $t=23$ and $d=5$.

For $d=4$ or $d=5$ we have $\lambda_1\in\{110,229\}$ and $\lambda_1\in\{506,1033\}$ respectively. But in any case the stabilizer of a point from $Y_i$ to $A^{Y_i}$ has exactly two non-single-point orbits on $Y_i$, and the lengths of these orbits are $54$ and $64$ or $256$ and $270$ respectively for $d=4$ and $d=5$ (this can be verified, for example, in GAP \cite{gap}). Contradiction.

\smallskip

2.2. Let $(t+1)=2y$ and $(t-1)=2^{d-2}x$, where $xy=2^{d}\pm 1$. Then $2^{d-3}$ divides $x\mp 1$,
which again implies $d\le 6$. But for $3\le d\le 6$ there is no admissible $t$, a contradiction.

\smallskip

3. Let $T \simeq L_d(q)$, $d\ge 3$.
Then the number $n_0-1=q\frac{q^{d-1}-1}{q-1}$ can be odd only if $q$ and $d-1$ are both  odd.

First, note that the cases $d=3$ and $d=5$ are impossible. Indeed, since $t^2-q^2\ne q+2$ for $t\ge 6$, then $d\ne 3$.
Now we show that $d\ne 5$. Suppose the contrary.
Since $d$ is odd, then $q\ne 2$ and therefore $q^2<t<q^3$. Let us expand $t$ in base $q$:
$t={\alpha}q^2+{\beta}q+\gamma$, where $0<\alpha, \gamma<q$ and $0\le \beta< q$. Then $t^2=\alpha^2q^4+2\alpha\beta q^3+(2\alpha\gamma+\beta^2)q^2+2\beta\gamma q+\gamma^2=q^4+q^3+q^2+q+2$ and therefore $\alpha=1$. If $\beta=0$, then $q$ divides the number $\gamma^2-2$ that is less than $q^2$, and then $\gamma=\sqrt{ql+2}$ for some non-negative integer $l$. But this leads to the contradiction %$2\sqrt{2+ql}(q^2+q)+q(l-1)=q^3+q^2$.
$2\sqrt{2+ql} +(l-1)/(q+1)=q$.

Further, arguing as in Case 1, in view of the relation \eqref{eq1*} we have $$\lambda_1=z_1\frac{n_0-1}{(2,n_0-1)}-t+\frac{t-1}{r},$$
where $z_1\in\{ 1,2,3,4\}$ for even $n_0-1$ and $z_1\in\{1,2\}$ for odd $n_0-1$.

Let us consider a vector space $V$ of dimension $d$ over $F_q$ and identify $T$ with $PSL(V)$.
Fix an arbitrary basis $u,w_1,w_2,...,w_{d-2}$ for $V$ and denote by $W_1$ the subspace generated by the vectors $w_2,...,w_{d-2}$. Recall that the stabilizer of two projective points $\langle u\rangle$ and $\langle w_1\rangle$ in $PSL(V)$ contains a subgroup isomorphic to $GL(W_1)$, which on $PV$ has one orbit of length $q-1$ (forming points of the set $P\langle u,w_1\rangle-\langle u\rangle-\langle w_1\rangle$), one orbit of length $\frac{q^{d-2}-1}{q-1}$ (forming points of  $PW_1$), and $q+1$ orbits of length $q^{d-2}-1$, that are entirely contained in $PV- PW_1-P\langle u,w_1\rangle$. Hence, identifying $\Sigma$ with the cartesian square of the projective space $PV$ for $PSL(V)$, we obtain
$$\lambda_1=\alpha_1+\alpha_2(q-1)+\alpha_3\frac{q^{d-2}-1}{q-1}+\alpha_4(q^{d-2}-1)$$
for some $\alpha_1\in\{0,1\}$, $\alpha_2,\alpha_3\in\{0,1,2\}$, $\alpha_4\in\{0,1,\ldots,2(q+1)\}$.

Thus
$$\frac{z_1q(q^{d-1}-1)}{(2,n_0-1)(q-1)}-t+\frac{t-1}{r}=\alpha_1+\alpha_2(q-1)+\alpha_3\frac{q^{d-2}-1}{q-1}+\alpha_4(q^{d-2}-1).$$

Hence
\begin{equation} \label{eqnum1}
\frac{z_1q}{(2,n_0-1)}(q^{d-2}-1+\frac{q^{d-2}-1}{q-1}+1)-t+\frac{t-1}{r} -\alpha_1-\alpha_2(q-1)=\alpha_3\frac{q^{d-2}-1}{q-1}+\alpha_4(q^{d-2}-1).
 \end{equation}
Multiplying both sides of the equality \eqref{eqnum1} by $(2,n_0-1)$, we obtain
\begin{equation} \label{eqnum2}
z_1q(q^{d-2}-1+\frac{q^{d-2}-1}{q-1})-\gamma=(2,n_0-1)\big(\alpha_3\frac{q^{d-2}-1}{q-1}+\alpha_4(q^{d-2}-1)\big),
\end{equation}
where
\begin{equation} \label{eqgam}
\gamma=-z_1q+(2,n_0-1)(t-\frac{t-1}{r}+\alpha_1+\alpha_2(q-1))=l\frac{q^{d-2}-1}{q-1}
\end{equation}
for some integer $l$.
Then, dividing both sides of the equality \eqref{eqnum2} by $\frac{q^{d-2}-1}{q-1}$, we obtain
\begin{equation} \label{eqnum3}
z_1q^2 =z_1q\cdot(q-1)+z_1q =l+(2,n_0-1)(\alpha_3+\alpha_4(q-1)),
\end{equation}
therefore, the number $l+(2,n_0-1)\alpha_3-z_1$ is divisible by $q-1$.

\smallskip
  
3.1. Let $l=\gamma=0$.
Then the number $(2,n_0-1)\alpha_3-z_1$ is divisible by $q-1$. Therefore $q\le 5$ or $z_1=(2,n_0-1)\alpha_3$ and $\alpha_4=\frac{(q+1)z_1}{(2,n_0-1)}$.

On the other hand, $\frac{z_1}{(2,n_0-1)}\le 2$ regardless of the parity of $n_0-1$, hence the expression \eqref{eqgam} yields
\begin{equation} \label{eqnum4}\frac{t-1}{r}(r-1)=z_1\frac{q}{(2,n_0-1)}-1-\alpha_1-\alpha_2(q-1) \le 2q-1.
\end{equation}
Therefore, $t-1\le(2q-1)2q$ and $$\frac{q^d-1}{q-1}=t^2-1\le (2q(2q-1))^2+4q(2q-1)=16q^4-16q^3+12q^2-4q,$$  which gives
$d\le 8$, and also that $q\le 16$ for $d\ge 6$.
Straightforward enumeration of admissible $d\ge 6$ such that $(t^2-1)(q-1)=q^d-1$ shows that $q=2$ and either $d=6$, $t=8$, and $r=7$, or $d=8$, $t=16$, and $r=5$, a contradiction with the estimate \eqref{eqnum4}.%, or $d=10$, $t=32$, and $r=31$.

Thus $d=4$, $z_1=2\alpha_3\in \{2,4\}$ and $\alpha_4\in \{q+1,2(q+1)\}$.

\smallskip

3.2. Suppose that $\gamma={\frac{q^{d-2}-1}{q-1}}l$ for some nonzero integer $l$.

\smallskip

3.2.1. Suppose that $l<0$. Then $-\gamma<z_1q$, which implies either $d=5$ and $q=2$ (which is impossible since $t^2\ne 2^5$), or $d=4$ and $z_1>-l$.

On the other hand, the equation \eqref{eqnum3} implies $z_1q^2<(2,n_0-1)(\alpha_3+\alpha_4(q-1))$, that is, $z_1<2(2,n_0-1)$.

Since the number $z_1-l-(2,n_0-1)\alpha_3$ is divisible by $q-1$, it cannot be negative
(otherwise $4>(2,n_0-1)\alpha_3>z_1-(-l)\ge 2$ and $q=2$, which is impossible).
But if $d=4$, then $z_1-l-(2,n_0-1)\alpha_3\le 2z_1\le 8,$ hence $q\le 9$, but then $t^2=q^3+q^2+q+2$ and there is no admissible $t$.
Contradiction.

\smallskip

3.2.2. Let $l>0$. Since $t\le q^{d/2}$, then    the identity \eqref{eqgam} implies that
$$(2,n_0-1)(q^{d/2}+1+2q-1)>l\frac{q^{d-2}-1}{q-1}+z_1q,$$ hence $d\le 7$.
If $d=7$, then the same identity gives $q<4$, hence $q=3$ and $t^2-1=1093$, a contradiction.

For $d=6$,  similarly we obtain $l=1$ and $(2,n_0-1)=2$. But the number $1+2\alpha_3-z_1$ is divisible by $q-1$, therefore $q\le 5$ or $z_1=1+2\alpha_3$. If $q\le 5$, then $q=2$, $t=8$ and $r=7$, but then by definition \eqref{eqgam} the number $\gamma$ must be odd, a contradiction.

Then, in view of the relation \eqref{eqgam}
$$2t=2\frac{t-1}{r}+z_1q+q^3+q^2+q+1-2(\alpha_1+\alpha_2(q-1))>q^3.$$
On the other hand, for $q\ge 6$ we have $t+1\le q^3$ and $t-1\ge \frac{q^3-1}{q-1}$, which implies
$$4t^2+t^2-(q-1)\le(q-1)t^2-(q-1)=(q-1)(t^2-1)=q^6-1<q^6$$
and $2t<q^3$, a contradiction.
Hence $q=2$ and $t=8$, a contradiction with the fact that $\gamma$ is odd.

Hence $d=4$.

The proof is complete.

\section{Proof of Theorem~\ref{MainThm:ImprimitiveCase}}
\label{ImprimitiveCase}

Let $G^{\Sigma}$ be an imprimitive group of rank $3$ with subdegrees $1$, $k_1$, and $k_2$, and $\Sigma_i$ be the set of fibres intersecting $X_i$, where $i=1,2$.
Denote by $\mathcal{B}$ the unique nontrivial imprimitivity system of $\tG:=G^{\Sigma}$ (see, for example, \cite[Lemma 3.3]{DGLPP}) and fix an arbitrary block $B$ (of size $k_1+1$).
Let $m=|\mathcal{B}|$ and $m_1=k_1+1$,
so that $m_1$ divides $(k_2,(t^2-1)^2)$.

Let $H={{\tG_{\{B\}}}}^B$ and $T=\soc(H)$.
Then $\tG\le H\wr \sym(\mathcal{B})$, $H$ acts $2$-transitively on $B$
and $\tG^{\mathcal{B}}$ acts $2$-transitively on $\mathcal{B}$ (see, e.g., \cite[Lemma 3.1]{DGLPP}).

Next we introduce some additional technical notation.

For  $A\subseteq \cF(\Gamma)$ we denote by $\overline{A}$  the set of all vertices
contained in fibres from $A$, and for  $X\subseteq\cV(\Gamma)$ we denote by $F(X)$   the set of all fibres containing a vertex from $X$.

Let $W=G_{\{\overline{B}\}}$, $S$ be the kernel of the action induced by $G$ on $\mathcal{B}$ (i.e., $S$ is the full preimage of $\tG\cap H^m$ in $G$), and $N$ be the kernel of the action induced by $W$ on ${B}$. Without loss of generality, we further assume that $a\in F\in B$.
Then $SN\trianglelefteq W$, $K\le N\cap S\trianglelefteq W$, and $N=K:N_a$.
\medskip

Note that $\cA(\Gamma)$ is the union of two self-paired orbitals of $G$, say, $Q_1$ and $Q_2$, with $|Q_i(a)|=k_i$, where $i=1,2$. In this case, the graph $\Gamma_i=(\cV(\Gamma), Q_i)$ can be considered as an undirected arc-transitive graph. It is edge-regular with parameters $(v, k_i,\lambda_i)$ and is an $r$-cover of its quotient $\Phi_i$ on $\Sigma$. More precisely, for each block $B'\in\mathcal{B}$, the subgraph $\overline{B'}$ of $\Gamma_1$ is an $r$-cover of the complete graph on $B'$, and $\Phi_1$ is the union of $m$ isolated $m_1$-cliques, and $\Gamma_2$ is an $r$-cover of $\Phi_2$, the complement of $\Phi_1$, that is, a complete $m$-partite graph with parts of order $m_1$. \begin{lemma}\label{ImprimitiveCase:Lem1} We have either
\begin{itemize}
 \item[$(i)$] $\lambda=\lambda_1=\lambda_2\le k_1-1$ and $k_1+1>4(t+1)^2$, or
\item[$(ii)$] $\lambda\not\in\{\lambda_1,\lambda_2\}$,
$k_1+1$ divides $\lambda-\lambda_1$, $(m-1,k_1)\ne 1$ and $k_1$ divides $(m-1,k_1)(\lambda-\lambda_2)$.
\end{itemize}
 \end{lemma}
\begin{proof}

Obviously, $\lambda_1\le k_1-1$, $\lambda=(t^3-2t+1)\frac{(t-1)}{r}+t(t^2-3)>4(t+1)^2$ and due to the equality \eqref{eq1}
$$k_1(\lambda-\lambda_1)=(k_1+1)(m-1)(\lambda-\lambda_2).$$

Now assume that $\lambda\not\in\{\lambda_1,\lambda_2\}$. Since $r\ge 5$, then
$$2\lambda=t(t^2-2)\frac{2(t-1)}{r}+ \frac{2(t-1)}{r}+2t(t^2-3)<t(t^2-2)\frac{(t-1)}{2}+2t(t^2-2)<t^2(t^2-2),$$
which implies
$$2\lambda-\lambda_1-\lambda_2<k_1+k_2=t^2(t^2-2).$$
Therefore $(k_1,k_2)=(m-1,k_1)\ne 1$ whenever $\lambda\not\in\{\lambda_1,\lambda_2\}$.

This implies the required assertion.

\end{proof}

\subsection{Reduction to  affine groups $H$}

Here we aim to prove the following result.

\begin{theorem}\label{ImprimitiveCase:Thm1} The group $H$ is not almost simple.

\end{theorem}
\begin{proof}
Suppose that the group $H$ is almost simple. Since $n>24$, then
in view of \cite[Theorem 1.2, Corollary 1.4]{DGLPP} we have either
\begin{itemize}
\item[$(i)$] $T^{m}\le S^{\Sigma}\le \tG$ (and $\tG$ is full), or

\item[$(ii)$] $S=K$ (that is, $\tG\cap H^m=1$ and $\tG\simeq \tG^{\mathcal{B}}$) and $\tG$ is an almost simple group.
\end{itemize}
\smallskip

Let us consider \underline{case $(i)$} and assume that $T^{m}\le \tG$.
Note that the groups $N_a^{X_2}$ and $N^{\Sigma_2}$ are permutation isomorphic, and $N_a\simeq N_a^{X_2}\simeq N^{\Sigma}$. Hence the group $N^{\Sigma_2}$ contains a subgroup $T^{m-1}$ of $S^{\Sigma_2}$ whose image $L$ in $N_a$, on the one hand, fixes the set $X_1$ pointwise, and on the other hand, fixes  setwise  $m-1$  blocks $B'$ on $\Sigma_2$ so that $L^{B'}\simeq T$.

Further, $L$ fixes the set $X_1$ pointwise and acts transitively (and even $2$-transitively for $(T,m_1)\neq (L_2(8),28)$) on each block $B'\in\mathcal{B}-\{B\}$. Moreover, the set of such blocks $B'$ gives an imprimitivity system of the group $G_a$ on $X_2$ whose blocks are the sets $[a]\cap \overline{B'}$.

Suppose that $(T,m_1)\neq (L_2(8),28)$ and $[a]\cap \overline{B'}$ contains an edge for some block $B'\in\mathcal{B}-\{B\}$. Then $[a]\cap \overline{B''}$ is an $m_1$-clique for all $B''\in \mathcal{B}-\{B\}$. This together with edge regularity of the graph $\Gamma_1$  implies $\lambda_1=k_1-1$ and
\begin{equation}\label{x0}
k_1+1\le \mu\frac{r}{t-1}
\end{equation}
by the bound \eqref{HofmanClique}. Since $\lambda>\mu$, by Lemma \ref{ImprimitiveCase:Lem1} we obtain $\lambda\not\in\{\lambda_1,\lambda_2\}$ and
\begin{equation}\label{x1}
k_1+1 \text{ divides }\lambda+2.
\end{equation}

If $[a]\cap \overline{B'}$ is a coclique for any block $B'\in\mathcal{B}-\{B\}$, then the bound \eqref{HofmanCoclique} gives $(k_1+1)\mu\le n(t-1)$,
therefore $\mu\le (t-1)m$ and $(t^2+t-1)(t-1)/r\le m$.
On the other hand, each vertex in $\overline{B'}\cap[a]$ is adjacent on average to $\frac{k_1}{r-1}$ vertices in $[a^*]\cap \overline{B'}$ for all $a^*\in F(a)-\{a\}$, so in this case
\begin{equation}\label{x2}
\frac{k_1}{r-1}\le \mu.
\end{equation}

We now show that $k_2>\mu$ or $(T,m_1)= (L_2(8),28)$.
Suppose that $k_2\le \mu$. If $(T,m_1)\neq (L_2(8),28)$, then, as was proved above, one of the relations \eqref{x1} or \eqref{x2} holds. In the corresponding cases we get
$$n=k_1+1+k_2\le \lambda+2+\mu<\lambda+2+(r-1)\mu=n,$$ or
$$n=k_1+1+k_2\le (r-1)\mu+1+\mu<(r-1)\mu+\lambda+2=n,$$ a contradiction.

Further, we have $T^{d}\simeq L/C_L(K)\le \aut(K)$ for some $0\le d\le m-1$. Moreover, $\max(\pi(T))\le t-1$ and $T$ has a nontrivial permutation representation on $K$ of degree $r_1$ not exceeding $r-1$ when $L\ne C_L(K)$.

Since $\aut(K)\le \sym_{t-1}$, then for any $p\in \pi(|T|)$ the number $(|T|_p)^d$ divides $(t-1)!$. On the other hand, $((t-1)!)_p<p^{\frac{t-1}{p-1}}$ (see, e.g., \cite[Exercise 2.6.8]{DM}), which implies $d<t-1$. Therefore, for $m\ge t$ for some $B'\in\mathcal{B}$ the group $C_L(K)$ contains the pointwise stabilizer $L_{Y}$ of the set $Y=X_2-\overline{B'}$ in $L$, and $L_Y\simeq T$.

Suppose that $m<t$. Then $k_1+1>t> r-1$ and $k_2<(t-1)(k_1+1)$. For $k_1+1>r_{\min}$, where $r_{\min}$ denotes the degree of the minimal permutation representation of $T$, one of the following possibilities holds for $(T, r_{\min})$ (see \cite{Mazurov1993} and \cite{Atlas}):
\begin{itemize}
\item $(T,r_{\min})=(L_2(5),5), (L_2(7),7), (L_2(9),6),(L_2(11),11)$ or $(L_4(2),8)$, where $k_1+1=6, 8, 10, 12$ or $15$, respectively;

\item $(T,r_{\min})=(Sp_{2l}(2),2^{2l-1}-2^{l-1})$, $l\ge 3$, where $k_1+1=2^{2l-1}+2^{l-1}$;

\item $(T,r_{\min})=(U_3(5),50)$, where $k_1+1=126$;

\item $(T,r_{\min})=(HiS,100)$, where $k_1+1=176$.
\end{itemize}
Hence if $L\ne C_L(K)$ we obtain $r_{\min}\le r-1<k_1+1 \le 3r$ and $k_2<3r(t-1)$, which implies
$n< 3rt$, a contradiction.

Hence $C_L(K)=L$ or $m\ge t$, so we may assume that $L_{Y}\le C_L(K)$.
As $C_L(K)=L_F$, it follows that
 there is an involution $h$ in $L_{Y}$ that fixes pointwise both the set of $m-1$ blocks from $\mathcal{B}$, including the block $B$, and the set $\overline{B}$.
Then $|\fixs(h)|\ge k_2$ and by Lemma \ref{Prel:Lemma5} we have either $k_2\le \mu$ and $(T,m_1)= (L_2(8),28)$,
or $k_2\le \lambda$, every non-single-point $\langle h\rangle$-orbit is an edge and, in particular, for some block $B'\in \mathcal{B}-\{B\}$ the set $X_2\cap \overline{B'}$ contains an edge (and therefore induces a clique if $(T,m_1)\neq (L_2(8),28)$).
In the second situation, for $(T,m_1)\neq (L_2(8),28)$, the relation \eqref{x0} leads to the contradiction
$$n=k_1+1+k_2\le \mu+1+\lambda<(r-1)\mu+\lambda+2=n.$$
Let us consider the first situation. As was proved earlier, we may assume that $\lambda_1< k_1-1$. If the relation \eqref{x2} is satisfied, then
$$n=k_1+1+k_2\le (r-1)\mu+1+\mu<(r-1)\mu+\lambda+2=n,$$ a contradiction.
Thus for some block $B'\in \mathcal{B}-\{B\}$ the set $X_2\cap \overline{B'}$ contains an edge
and, since the primitive group $T$ acts on $28$ points as a group of rank $4$ with three non-single-point suborbits of length $9$ (see, e.g., \cite[Lemma 2.10]{GLP}),
we obtain $9\le\lambda_1$ and $\mu\ge 9$. Moreover, $28$ divides $n$, so $t$ is odd and $t\ge 13$. Hence $\lambda-\mu=(t^2-3)t\ge 166\cdot13$ and
$$n=k_1+1+k_2\le \mu+18+1+\lambda <(r-1)\mu+\lambda+2=n,$$ which is impossible.

\underline{The case $(ii)$} is ruled out as by \cite[Theorem 1.3]{DGLPP} the group $H$ can be almost simple only for $m\le 73$ and $n$ cannot be a perfect square in any of the rest cases.

\end{proof}

\subsection{Case of  affine groups $H$}

Next we prove Theorems \ref{ImprimitiveCase:Thm2} and \ref{ImprimitiveCase:Thm3}, which describe the pair $(\Gamma,G)$ in the cases where $H$ is affine and $G^{\mathcal{B}}$ is almost simple or affine, respectively.
\begin{theorem}\label{ImprimitiveCase:Thm2}
If the group $H$ is affine of degree $p^l$, where $p$ is a prime, and the group $G^{\mathcal{B}}$ is almost simple, then $K<S$, $C_G(S)\le W$, $l$ is even, and either
\begin{enumerate}
 \item $(\soc(G^{\mathcal{B}}), m) = (L_2(8), 9)$ and $N_a\le S$, or
 \item $(\soc(G^{\mathcal{B}}),m)=(L_5(3),121)$,
 $t=12$, $r=11$, $p=13$, $l=2$, $\lambda_1=17$, $\lambda_2=3369$, and $M^B\le GL_2(13)$ or $M^B\le \Gamma L_1(13^2)$.

\end{enumerate}
\end{theorem}

\begin{proof}
Suppose that $H$ is an affine group and $G^{\mathcal{B}}$ is almost simple. Put $Y=\soc(G^{\mathcal{B}})$. Then $|B|=k_1+1=p^{l}$ for a prime $p$ and $T$ is an elementary abelian group of order $p^l$.

Obviously, the pair $(|B|;|\mathcal{B}|)=(k_1+1;m)$ satisfies one of the following three conditions.
\begin{itemize}
\item[(C1)] $p=2$. Then $t$ is odd, $(t-1,t+1)=2$ and since by condition $t-1$ cannot be a power of $2$, then $(t-1)_{2'}>1$ and $(t^2-1)^2_{2'}$ divides $m$, in particular, $r^2$ divides $m$.
In this case
\begin{itemize}
\item[(C1.1)] $(r,k_1+1)=1$, $k_1+1$ divides $(2(t+1)_2)^2$ or
\item[(C1.2)] $(r,k_1+1)=1$, $k_1+1$ divides $(2(t-1)_2)^2$;
\end{itemize}
\item[(C2)] $p>2$, $(r,k_1+1)=1$ and $k_1+1$ divides $(t+1)^2_{2'}$, and $m$ is divisible by $(t-1)^2(t+1)^2_2$;
\item[(C3)] $p>2$ and $k_1+1$ divides $(t-1)^2_{2'}$, and $m$ is divisible by $(t+1)^2(t-1)^2_2$.
\end{itemize}
\medskip

Further, in view of Theorem \ref{2tra}, one of the following possibilities is admissible for $(Y, m)$:
\begin{itemize}
\item[(i)] $(Y,m)=(\alt_m,m)$, $m\ge 5$;

\item[(ii)] $(Y,m)=(Sp_{2d}(2),2^{2d-1}\pm 2^{d-1})$, $d\ge 3$;

\item[(iii)] $(Y,m)=(L_d(q),\frac{q^d-1}{q-1})$, $q$ is a power of a prime $p_0$, $d\ge 2$, $(d,q)\ne (2,2), (2,3)$;

\item[(iv)] $(Y,m)=(Sz(q),q^2+1)$, $q=2^{2e+1}\ge 8$;

 \item[(v)] $(Y,m)=(U_3(q),q^3+1)$, $q\ge 3$ is a power of a prime $p_0$;

 \item[(vi)] $(Y,m)=({^2G}_2(q),q^3+1)$, $q=3^{2e+1}\ge 27$;
 \item[(vii)] $(Y,m)=(M_m,m)$, $m\in\{11,12,22,23,24\}$;

 \item[(viii)] $(Y,m)=(L_2(11),11), (M_{11},12), (\alt_{7},15), (L_2(8),28),(HiS, 176)$, or $(Co_3,276)$.

\end{itemize}
\medskip

The  \underline{cases $(vii)$ and $(viii)$ for $Y$} are immediately ruled out, as straightforward enumeration shows that $m$ must be even, $p>2$ and $t\le 15$, while $n/m$ cannot be a prime power.

Therefore $(Y,m)\ne(L_2(8),28)$ and hence the group $Y$ is $2$-transitive on $\mathcal{B}$.
Taking into account that $k_1+1<4(t+1)^2<\lambda$, we obtain $\lambda<k_2$ (otherwise $2\lambda>(t^2-1)^2$, which is impossible) and by Lemma \ref{ImprimitiveCase:Lem1} $k_1>1$ and $(m-1,k_1)\ne 1$. Hence $G_{a,(X_2)}=1$ and $G_a^{X_2}\simeq G_a$.

\smallskip

\begin{stat}\label{Statement1}
$N_a$ acts intransitively on ${\mathcal{B}}-\{B\}$, $|\fix_F(N_a)|=|N_K(N_a)|$
divides $r$, and if $|\fix_F(N_a)|=1$, then $\lambda_1=k_1-1$.
\end{stat}
\begin{proof}Suppose that $N_a$ acts transitively on $\mathcal{B}-\{B\}$.
Since $N_a$ is normal in $G_a$, then the length of each $N_a$-orbit on $X_2$ is $(m-1)p^{\alpha}$ for some $0\le \alpha\le l$.
Hence $\lambda-\lambda_1=(m-1)p^{\alpha}f$ for some integer $f$. Then by Lemma \ref{Prel:Lemma6}
$$(p^l-1)(m-1)p^{\alpha}f=(m-1)p^l(\lambda-\lambda_2),$$
whence $p^lf-f=p^{l-\alpha}(\lambda-\lambda_2)$ and $p^{l-\alpha}$ divides $f$.
This means that $\lambda-\lambda_1=k_2$, a contradiction.

Since $N=K:N_a$, we have $|\fix_F(N_a)|=|N_K(N_a)|$. Further, for any neighbour $x\in \overline{B}-F$ of $b\in X_1$ and for any element $g\in N_a$, we have $x^g\in F(x)\cap [b]$, which implies $x^g=x$. Therefore, if $|\fix_F(N_a)|=1$, then $\lambda_1=k_1-1$.
\end{proof}
\medskip

\begin{stat}\label{Statement2}If $\lambda_1=k_1-1$, then $k_1+1$ divides $(\lambda+2,n)=\frac{(t-1)^2}{r}(t+1,r-1)$.

\end{stat}
%\texttt{Proof:}
\begin{proof}
By Lemma \ref{ImprimitiveCase:Lem1}, $k_1+1$ divides $(\lambda+2,n)$ when $\lambda_1=k_1-1$.
It remains to note that $(\lambda+2,n)=\frac{(t-1)^2}{r}(t+1,r-1)$.
\end{proof}
\medskip

\begin{stat}\label{Statement3}
$(m-1,k_1)$ divides $(k_1,k_2)$ and $(t^2(t^2-2),k_2)$.
\end{stat}
\begin{proof} Obvious.
\end{proof}

By Theorem \ref{2tra}, for $W^B=\soc(W^B):M^B$ and $k_1+1=p^l$ one of the following possibilities holds:
\begin{itemize}
\item[(a)] linear $l=cd$, $d\ge 2$ and $SL_d(p^c)\trianglelefteq M^B\le \GamL_d(p^c)$;

\item[(b)] symplectic $l=cd$, $d\ge 4$, $d$ even and $Sp_d(p^c)\trianglelefteq M^B\le Z_{p^c-1}\circ {\Gamma}Sp_d(p^c)$;

\item[(c)] $G_2$ of type $l=6c$, $p=2$ and $G_2(2^c)'\trianglelefteq M^B\le Z_{2^c-1}\circ \aut(G_2(2^c))$;
\item[(d)] one-dimensional $M^B\le \GamL_1(p^l)$;

\item[(e)] exceptional $p^l\in\{9^2,11^2,19^2,29^2,59^2\}$ and $SL_2(5)\trianglelefteq M^B$, or
$p^l=2^4$ and $M^B \trianglerighteq \alt_6$ or $\alt_7$, or $p^l=3^6$ and $SL_2(13)\trianglelefteq M^B$;

\item[(h)] extraspecial $p^l\in \{5^2,7^2,11^2,23^2\}$ and $SL_2(3)\trianglelefteq M^B$,
 or $p^l=3^4$, $R=D_8\circ Q_8\trianglelefteq M^B$, $M^B/R\le \sym_5$ and $5$ divides $|M^B|$.

\end{itemize}
Note that in the cases $(b), (c), (e)$ or $(h)$ for $M^B$ the number $l$ is always even.

\begin{stat}\label{Statement4}If $l$ is even and  the case $(iii)$ with $d\ge 3$ holds for $Y$, then $t=12$, $m=121$, $r=11$, $p=13$, $l=2$, $\lambda_1=17$, $\lambda_2=3369$, and $Y=L_5(3)$, where $M^B\le GL_2(13)$ or $M^B\le \Gamma L_1(13^2)$.

\end{stat}

\begin{proof}
Note that if $l$ is even (as in cases $(b), (c), (e)$, or $(h)$ for $M^B$), then $\frac{q^d-1}{q-1}=\frac{(t^2-1)^2}{p^l}$ is a perfect square.
But by  \cite{Ljunggren}  the Nagell--Ljunggren      equation $\frac{x^i-1}{x-1}=y^2$, where $i>2$, is unsolvable for $|x|>1$, except for the cases when
(1) $x=7$, $i=4$ or (2) $x=3$, $i=5$.
Hence the pair $(Y,\frac{(t^2-1)}{p^{l/2}})$ is, respectively, $(L_4(7),20)$ or $(L_5(3),11)$.

If $\sqrt{m}=\frac{(t^2-1)}{p^{l/2}}=11$, then $p\ne 2$ and $t\le 12$.
Since $t-1$ is not a power of $2$ or $3$, we obtain $t=12$, $p=13$ and $l=2$, which by Lemma \ref{ImprimitiveCase:Lem1} implies
$\lambda_1=17$, $\lambda_2=3369$.
Then $M^B\le GL_2(13)$ or $M^B\le \Gamma L_1(13^2)$.

Let $\sqrt{m}=\frac{(t^2-1)}{p^{l/2}}=20$.
If $p$ divides $20$, then $t^2-1=2^i5^j$ for some $i,j$. Since $t-1$ is not a power of 2, then either $p=5$ and $t-1=5^{l/2}$, or $p=2$ and $(t-1)/2=5^{l/2}$ and $4(t+1)=2^i$. The first case leads to the contradiction $t+1=2^i=5^j+2$.
In the second, we have $t-1=2\cdot 5^{j}=2(2^{i-1}-1)$, which by Lemma \ref{Prel:LemmaZsigmCor} yields $j=1$, a contradiction.

Therefore $p^{l/2}$ divides one of $t-1$ and $t+1$, whence $t\le 21$. Straightforward enumeration shows that there is no admissible $t$.

\end{proof}

\medskip
\begin{stat}\label{Statement5}
Suppose that $l$ is even. Then cases $(iv, v, vi)$ for $Y$ are impossible and either $(Y,m)=(L_2(8),9)$ and $N_a\le S_a$ or $m-1$ is not a prime power.
\end{stat}
\begin{proof}
Let $l$ be even. Then $m=\frac{(t^2-1)^2}{p^l}=y^2$ is a perfect square. If $m-1$ is a power of a prime $p_0$,
then $p_0=2$ and $m=9$.
Hence we have either $(Y,m)=(L_2(8),9)$ or $m-1$ is not a prime power. This rules out the cases $(iv, v, vi)$ for $Y$.
In the first case, $Y$ acts $3$-transitively on $\mathcal{B}$, so $N_a\le S_a$ by Claim \ref{Statement1}.
\end{proof}
\begin{stat}\label{Statement6}

If $l$ is odd, then $p$ divides $m$, and $p^l<m$ for all $(p;m)\ne (2; (t-1)^2/2)$.
\end{stat}
\begin{proof}
Let $l$ be odd. Since $mp^l=(t^2-1)^2$ is a perfect square, $p$ divides $m$ and $(p,m-1)=1$. Moreover, $t\equiv\pm 1 \pmod {p_1}$ for all $p_1\in\pi(m)=\pi(n)$ and $\pi(r)\subseteq\pi(m)\cap \pi(t-1)$.

\medskip
\begin{itemize}
 \item[1)] Let $p>2$. Then $p^l$ divides one of $(t+1)^2$ or $(t-1)^2$.

\begin{itemize}
 \item[1.1)] If $p^l$ divides $(t-1)^2$, then $p^l<m=(t+1)^2\frac{(t-1)^2}{p^l}$.

\item[1.2)] If $p^l$ divides $(t+1)^2$, then again $p^l<m=p(t-1)^2\frac{(t+1)^2}{p^{l+1}}$,
since obviously $p(t^2-2t+1)>t^2+2t+1$ for $t\ge 6$.
\end{itemize}
\item[2)] Let $p=2$. Then $p^l$ divides one of $4(t+1)^2$ or $4(t-1)^2$.
\begin{itemize}
 \item[2.1)] If $p^l$ divides $4(t-1)^2$, then $p^l<m=r^2\frac{(t+1)^2}{4}\frac{4(t-1)^2}{r^2p^l}$.

\item[2.2)] If $p^l$ divides $4(t+1)^2$, then either $p^l=2(t+1)^2$, or $p^l\le(t+1)^2/2$ and $p^l<m=2(t-1)^2\frac{(t+1)^2}{2p^l}$.
\end{itemize}
\end{itemize}
\end{proof}

Next we consider the cases $K=S$ and $K<S$ separately.

\medskip
\texttt{1.} If $K=S$ (that is, $S^{\Sigma}=1$ and $\tG$ is quasi-primitive), then according to \cite[Theorem 1.3]{DGLPP}
either $Y\simeq L_2(q)$, $T\simeq Z_2$ and $q+1=m$,
or $Y\simeq L_d(q)$, $W^B\simeq AGL_1(p^l)$ and $\frac{q^d-1}{q-1}=m$, $d\ge 3$, $l=1$ and $(k_1+1,q-1)=k_1+1$.
The condition $(m-1,k_1)\ne 1$ excludes the first case. In the second case we have $m-1=q\frac{q^{d-1}-1}{q-1}$, $k_1=p-1$
and $p$ divides $m$.

If $p=2$, then $M=N$ and $N_a$ acts transitively on $X_2$, which implies $\lambda>k_2$, a contradiction.
Hence, $p>2$ and since $(t^2-1)^2=pm$, then by Lemma \ref{Prel:LemmaNT} $1\ne (m)_p=(d)_p=p^{2i+1}$ for some $i\in \mathbb{N}$.

\smallskip

\texttt{1.1.} If $C_G(K)\le K$, then $Y\le \aut(K)$ has a faithful permutation representation of degree at most $r-1<(t-1)_{2'}\le m$,
a contradiction with \cite{Mazurov1993}.

\smallskip

\texttt{1.2.} If $C_G(K)\not\le K$, then $\soc(G/S)\le C_G(K)/K$ acts 2-transitively on $\mathcal{B}$ and $C_W(K)$ acts transitively on $\mathcal{B}-\{B\}$.

Let $R$ be the full preimage of $Y$ in $C_G(K)$. Then $R^{\infty}/Z(R^{\infty})=R^{\infty}/(R^{\infty}\cap K)\simeq R^{\infty}K/K=R/K$, $R=R^{\infty}K$ and $R^{\infty}$ is a quasi-simple covering group for $Y$. Therefore
either $R^{\infty}=R'\simeq Y$ and $R=R'\times K$, or $R/Z(R^{\infty})=Y\times (K/Z(R^{\infty}))$.
But then the group ${R^{\infty}}\cap W$  contains a nonsolvable subgroup of the form $E_{q^{d-1}}.SL_{d-1}(q)$ (see, for example, \cite{Mazurov1993}), which acts transitively on $\mathcal{B}-\{B\}$, and since $W^B\le AGL_1(p)$ and $p$ divides $q-1$, it is contained in $N$.
From this, by \cite[5.20]{Asch}, we get that the group
$N_a$ also acts transitively on $\mathcal{B}-\{B\}$, a contradiction with Claim \ref{Statement1}.

\smallskip

\texttt{2.} Let $K<S$. Then $S\not\le N$, $Q:=\soc(W/N)\le SN/N\simeq S/N\cap S$ and $Q=C_{W/N}(Q)$.
Hence $Q=O_p(SN/N)\simeq T$, all $S$-orbits on $\Sigma$ have length $k_1+1$ and $W=G_aS$.
On the other hand, $W/S=MS/S=G_aS/S\simeq G_a/S_a$ and $M/N=G_aN/N\simeq G_a/N_a\simeq {M}^B\le \aut(Q)$.
\medskip

\begin{stat}\label{Statement7}
 $C_G(S)\le W$.
\end{stat}
 \begin{proof} Suppose $C_G(S)\not\le W$. Then
$C_G(S)S/S\trianglelefteq G/S$ and $\soc(G/S)\le C_G(S)S/S\simeq C_G(S)/Z(S)$.
Hence $C_G(S)$ acts $2$-transitively on ${\mathcal{B}}$ and
so  $C_W(S)$ (that is, the stabilizer of $B$ in $C_G(S)$) acts transitively on ${\mathcal{B}}-\{B\}$.

%Let us now show that $S_a=1$ and $S_{\{F\}}=K$.
 Since $C_G(S)$ acts $2$-transitively on $\mathcal{B}$, then, for $B'\in \mathcal{B}-\{B\}$, the stabilizer $U$ of $B'$ in $C_G(S)$ acts
transitively on $\mathcal{B}-\{B'\}$ and the $U$-orbit on $\Sigma$ containing $F(a)$ is entirely contained in $\fixs(S_a)$.
But all $S_a$-orbits on $\Sigma_2$ have the same length.
Therefore $S_a=1$ and $S_{\{F\}}=K$.

 Hence $W^{\mathcal{B}}\simeq G_a$.
On the other hand, $C_W(S)N/N$ is normal in $M/N$ and centralizes $SN/N$.
This implies that either $C_W(S)\le N$ or $C_W(S)/C_N(S)\simeq C_W(S)N/N=Q=SN/N$.

Let $\widetilde{Y}$ be the full preimage of $Y$ in $G$. Then $\widetilde{Y}\le SC_G(S)$ and $G_a/\widetilde{Y}_a\le \out(Y)$. But $\widetilde{Y}_a\le C_W(S)S$ and therefore $\widetilde{Y}_a$ centralizes the group $SN/N$,
which implies $\widetilde{Y}_a\le N_a$. But then $N_a$ acts transitively on ${\mathcal{B}}-\{B\}$,
which is impossible by Claim \ref{Statement1}.
\end{proof}

Thus by Claim \ref{Statement7} we have $C_G(S)\le W$. Then $S\ge C_G(S)=Z(S)$ and $Z(S)$ centralizes $SN/N$. Hence either %$C_G(S)\le N$ and
$Z(S)\le K$, or $Z(S)/C_N(S)\simeq Z(S)N/N=Q=SN/N$, $Z(S)N\ge S=(N\cap S).Q$ (so that $S_a\le N_a$ and $N\cap S=K:S_a$) and $K< C_G(K)$.

\medskip

In \underline{case (i) for $Y$} we have $\alt_{m-1}\le W^{\mathcal{B}}\le \sym_{m-1}$ and $W^{\mathcal{B}}$ acts $2$-transitively on ${\mathcal{B}}-\{B\}$, which implies $N_a\le S_a$ (otherwise $N_a$ would act transitively on ${\mathcal{B}}-\{B\}$, which  is impossible by Claim \ref{Statement1}).
Then $M^{B}\simeq G_a/N_a$ contains a normal subgroup $U$ isomorphic to $S_a/N_a$, so   $M^B/U$ is isomorphic to $W^{\mathcal{B}}\simeq G_a/S_a$.
Let $D=M^B$. Then $\alt_{m-1}\le D/U\le \sym_{m-1}$ and therefore $D$ is non-solvable.
Moreover, $(m-1)!$ divides $|M^B|\cdot 2$ divides $(k_1+1)!\cdot 2$. Hence, by Lemma \ref{ImprimitiveCase:Lem1} $m-1<k_1+1$ and, since $k_1+1$ divides $(t^2-1)^2$, we obtain $k_1+1>t^2+1$. Hence, in view of the conditions $(C1)$--$(C3)$, we have either $p>2$, $k_1+1=(t+1)^2$ and $m=(t-1)^2$, or $p=2$, $(t-1,4)=2$, $k_1+1$ divides $4(t+1)_2^2=4(t+1)^2$ and $(m)_{2'}=(t-1)^2/4$. If one of the cases $(a)$, $(b)$  or $(c)$ holds for $M^B$ (and correspondingly $D^{\infty}\simeq SL_d(p^c)$, $Sp_d(p^c)$ or $G_2(2^c)'$), then
$D/D^{\infty}$ is solvable, so $D^{\infty}\not\le U$ and $D^{\infty}/D^{\infty}\cap U\simeq \alt_{m-1}$.
But then the degree of the minimal permutation representation of $D^{\infty}/D^{\infty}\cap U\simeq \alt_{m-1}$
is equal to $m-1$. By \cite[Theorems 1, 2]{Mazurov1993} and \cite[Theorem 1]{Vasil'ev1996} we obtain $m-1=5,6,8$, which   implies $t=4$ or $7$, a contradiction.
Suppose the case $(e)$ or $(h)$ holds for $M^B$. If $t+1$ is a prime or $(t+1)^2=3^4, 3^6$, then $D\le GL_l(p)$ and
$\alt_{(t-1)^2-1}\le D/U$, which implies $m-1=5,6$, a contradiction.
If $k_1+1=2^4$, then $m-1\le 7$, which contradicts the fact that $(m)_{2'}=(t-1)^2/4$.

\smallskip

In \underline{case (ii) for $Y$} the number $t$ is odd, $\out(Y)=1$
and $W^{\mathcal{B}}\simeq PSO^{\pm}_{2d}(2)$, $d\ge 3$. Then $(W^{\mathcal{B}})'\simeq P\Omega^{\pm}_{2d}(2)$ is a non-abelian simple subgroup of index $2$ in $W^{\mathcal{B}}$.
Since $S_a N_a\trianglelefteq G_a$, $W^{\mathcal{B}}\simeq G_a/S_a$ contains a normal subgroup $S_aN_a/S_a$ isomorphic to $N_a/(S_a\cap N_a)$. %Hence either $N_a\le S_a$ or $|G_a:N_aS_a|\le 2$.
By Claim \ref{Statement1} $N_a$ acts intransitively on ${\mathcal{B}}-\{B\}$, so $(W^{\mathcal{B}})'\not\le (N_a)^{{\mathcal{B}}}$, and hence $|G_a:N_aS_a|>2$, which implies $N_a\le S_a$.
On the other hand, $D:=M^{B}\simeq G_a/N_a$ contains a normal subgroup $U\simeq S_a/N_a$, so that $D/U$ is isomorphic to $W^{\mathcal{B}}$. This immediately excludes the case $(d)$ for $M^{B}$. If one of the cases $(a), (b)$ or $(c)$ holds for $M^B$ (and correspondingly $D^{\infty}\simeq SL_s(q)$, $Sp_s(q)$ or $G_2(q)'$, where $q=p^c$ and $sc=l$), then
$D/D^{\infty}$ is solvable, so $D^{\infty}\not\le U$ and $D^{\infty}/(D^{\infty}\cap U)\simeq P\Omega^{\pm}_{2d}(2)$. Then (see, for example, \cite[Ch. 17]{Asch})
$D^{\infty}/(D^{\infty}\cap U)\simeq P\Omega^{+}_{6}(2)\simeq L_4(2)$, $s=4$ and $q=2$, that is, $k_1+1=16$
and $m=36$, which contradicts the condition $t\ge 6$.
 Cases $(e)$ and $(h)$ for $M^B$ are excluded similarly.

\smallskip

In \underline{case (iii) for $Y$ with $d\ge 3$} we have $m=\frac{q^d-1}{q-1}$,
$Y\simeq L_d(q)$ and
$W^{\mathcal{B}}$ contains a normal non-solvable subgroup $X:=(W^{\mathcal{B}})^{\infty}\simeq ASL_{d-1}(q)$ (see, e.g., \cite{Mazurov1993}), which acts transitively on $\mathcal{B}-\{B\}$.

By Claim \ref{Statement4} in cases $(b), (c), (e)$ or $(h)$ for $M^B$, and in cases $(a),(d)$ with even $l$ we have
$t=12$, $m=121$, $r=11$, $p=13$, $l=2$, $\lambda_1=17$, $\lambda_2=3369$ and $Y=L_5(3)$, where $M^B\le GL_2(13)$ or $M^B\le \Gamma L_1(13^2)$.

Let $t\ne 12$. Then $l$ is odd.

Suppose that $N_a\le S_a$. Then $D:=M^{B}\simeq G_a/N_a$ contains a normal subgroup $U\simeq S_a/N_a$, so that $D/U$ is isomorphic to $W^{\mathcal{B}}$. This immediately excludes the case $(d)$ for $M^{B}$. In case $(a)$ for $M^B$ with odd $l=c\tilde{d}\ge 3$ we have $D^{\infty}\simeq SL_{\tilde{d}}(p^c)$ and $ASL_{d-1}(q)\trianglelefteq W^{\mathcal{B}}$, a contradiction.

Further, since $S_a N_a\trianglelefteq G_a$, $W^{\mathcal{B}}\simeq G_a/S_a$ contains a normal subgroup $R:=(N_a)^{{\mathcal{B}}}$ isomorphic to the group $S_aN_a/S_a\simeq N_a/(S_a\cap N_a)$.
By Claim \ref{Statement1} $N_a$ acts intransitively on ${\mathcal{B}}-\{B\}$, so $X\not\le R$ and $S_aN_a\ne G_a$.

We have $N_a\not\le S_a$, that is, $R>1$. Further, $X=A:J$, where $A$ is a normal elementary abelian subgroup of order $q^{d-1}$ and $J\simeq SL_{d-1}(q)$. Moreover, on $\mathcal{B}-\{B\}$ there are exactly two $J$-orbits, say $J_1$ and $J_2$, of lengths $\frac{q^{d-1}-1}{q-1}$ and  $q^{d-1}-1$, respectively, so that $J$ acts $2$-transitively on the set of all $\frac{q^{d-1}-1}{q-1}$ $A$-orbits (of length $q$) on $\mathcal{B}-\{B\}$. Moreover, if $R\cap A>1$, then $A\le (R\cap A)^J\le R$.

If $R\cap X=1$, then $R$ centralizes $X$. But $C_{W^{\mathcal{B}}}(A)\cap Y\le A$, so $R\cap Y=1$ and $|R|$ divides $(d,q-1)e$. But then $J_1$ intersects each $R$-orbit in at most one point, hence the number of $R$-orbits is at least $\frac{q^{d-1}-1}{q-1}$, and each of them is contained in some $A$-orbit (since $|\fix_{\mathcal{B}-\{B\}}(A_1)|=q$ for some subgroup $A_1\le A$ of index $q$), and therefore its length divides $(q,e)$ and, consequently, $R\cap PGL_d(q)=1$.
But no nontrivial field automorphism of $Y$ centralizes $A$, so $R=1$, a contradiction.

This means $R\cap X>1$ and therefore $A\le R$.

Thus in the case $R\cap X>A$ we have $R\cap X\le A.Z(J)$.

Since $M^B\simeq G_a/N_a$ contains a normal subgroup $X_2$ isomorphic to the group $S_aN_a/N_a$,
then $$W^{\mathcal{B}}/R\simeq M^B/X_2=:L,$$
and   the group $L$ cannot be solvable except for the cases $(d;q)=(3;3),(3;2)$.

Suppose that $(d;q)\ne (3;3),(3;2)$.

The case $(d)$ for $M^B$ is immediately excluded due to the solvability of the group $M^B$.

In case $(a)$ for $M^B$ with odd $l=c\tilde{d}\ge 3$ we have $D:=({M^B})^{\infty}\simeq SL_{\tilde{d}}(p^c)\not\le X_2$, $D$ is non-solvable, $DX_2/X_2$ is isomorphic to a non-identity quotient group of $SL_{\tilde{d}}(p^c)$, and $L$ is isomorphic to a section of $\Gamma L_{\tilde{d}}(p^c)$.

On the other hand, $XR/R$ is isomorphic to a non-identity quotient group of $SL_{d-1}(q)$ and $W^{\mathcal{B}}/X\le Z_{(q-1)}.Z_e$.

Combining the above observations, we obtain
$$XR/R=(W^{\mathcal{B}}/R)^{\infty}\simeq L^{\infty}\simeq DX_2/X_2,$$
which yields $$L_{d-1}(q)\simeq L_{\tilde{d}}(p^c).$$
Therefore, this is one of the pairs $L_2(4)\simeq L_2(5)$ or $L_3(2)\simeq L_2(7)$ (see, for example, \cite[p. 254]{Asch}).
But it is impossible as $l=c\tilde{d}\ge 3$ and $d\ge 3$.

Then, for $(d;q)= (3;3),(3;2)$ we get $m=7$ or $13$, which implies
$p=m$. A contradiction with Claim \ref{Statement6}.

Finally, let $t=12$ and $l$ be odd. Then $p$ divides $m$, which yields $l=1$, and $p=11$ or $13$.
In any case, $M^B\le Z_{p-1}$, which  contradicts the fact $X\not\le R$.

\smallskip

Now we consider \underline{case (iii) with $d=2$ and cases (iv, v, vi) for $Y$}.

For even $l$, by Claim \ref{Statement5} we have $(Y,m)=(L_2(8),9)$ and $N_a\le S_a$.

Let $l$ be odd. Then in case $(a)$ for $M^B$ with $l=c\tilde{d}\ge 3$ we have $D:=({M^B})^{\infty}\simeq SL_{\tilde{d}}(p^c)$, and in case $(d)$ for $M^B$ we have $M^B\le \Gamma L_1(p^l)$ and $|M^B|$ divides $l(p^l-1)$.

Furthermore, $U\le W^{\mathcal{B}}\le N_{G^{\mathcal{B}}}(U)\le N_{Y}(U).Z_f.Z_e$ for $U\in \syl_{p_0}(Y)$, where $q=p_0^e$, $p_0$ is a prime, and $f$ is $(2,q-1)$, $(3,q+1)$, and $1$ in cases (iii), (v), and (iv,vi), respectively. In any of these cases, $N_Y(U)/U$ is a cyclic group of order dividing ${q^2-1}$.

Since $W^{\mathcal{B}}\simeq G_a/S_a$ contains a normal subgroup $X_1$ isomorphic to $S_aN_a/S_a$,
and $M^B\simeq G_a/N_a$ contains a normal subgroup $X_2$ isomorphic to $S_aN_a/N_a$,
we obtain $$W^{\mathcal{B}}/X_1\simeq M^B/X_2=:L.$$

By Claim \ref{Statement1}, $N$ is intransitive on $\mathcal{B}-\{B\}$, so $U\not\le X_1$ and $G_a\ne S_aN_a$.
This means that $L\ne 1$ and $L$ is solvable. Moreover, if $U\cap X_1=1$, then $X_1=1$.

\smallskip

\begin{itemize}

\item[$(1)$] Suppose $(Y,m)=(L_2(q),(q+1))$. Then $X_1=1$, so $N_a\le S_a$ and
$$N_Y(U)\le W^{\mathcal{B}}\simeq M^B/X_2=L\le N_Y(U).Z_2.Z_e.$$

Therefore, in case (a), we have $SL_{\tilde{d}}(p^c)\le X_2$ and the group
$M^B/X_2$ is isomorphic to a section of $\Gamma L_{\tilde{d}}(p^c)/SL_{\tilde{d}}(p^c)\simeq Z_{(p^c-1)}.Z_c$.

In case (d), $M^B/X_2$ is isomorphic to a section of $\Gamma L_1(p^l)\simeq Z_{(p^l-1)}.Z_l$.

But $N_Y(U)$ is a split extension of an elementary abelian group $U$ of order $q$ by a cyclic group of order $(q-1)/(2,q-1)$, which acts semiregularly on $U$.

Thus, in any case, $(m-1)(m-2)/(2,m-2)$ divides $(p^l-1)l$,
which obviously cannot hold for $k_1<m-1$.

Then by Claim \ref{Statement6}, $p=2$, $m$ is even, $q=m-1=(t-1)^2/2-1$, and $2^{l-1}=(t+1)^2$.
Hence $m+2t=(k_1+1)/4$ and the number $4q=k_1+1+8t-4$ divides $4k_1l$.
Since $k_1+1\ne 8t-4$, then $(k_1+1+8t-4,4k_1l)$ divides $(1+8t-4,4)l$.
But then $k_1+1=p^l\le l$, a contradiction.

\item[$(2)$] Suppose $(Y,m)=(Sz(q),q^2+1)$.

Then $X_1< U$, so either $X_1=1$ and $N_a\le S_a$, or $X_1=Z(U)$.
Both possibilities yield
$$N_Y(U)/X_1\le W^{\mathcal{B}}/X_1\simeq M^B/X_2=L\le (N_Y(U)/X_1).Z_e.$$

In case (a) we have $SL_{\tilde{d}}(p^c)\le X_2$ and the group
$M^B/X_2$ is isomorphic to a section of the group $\Gamma L_{\tilde{d}}(p^c)/SL_{\tilde{d}}(p^c)\simeq Z_{(p^c-1)}.Z_c$.

In case (d) the group $M^B/X_2$ is isomorphic to a section of the group $\Gamma L_1(p^l)\simeq Z_{(p^l-1)}.Z_l$.

But $UX_1/X_1$ is a non-cyclic $2$-group of order $q$ or $q^2$, which contradicts the fact that $N_Y(U)/X_1\le Z_{(p^l-1)}.Z_l$ in any case.

\item[$(3)$] Suppose $(Y,m)=(U_3(q),q^3+1)$ or $(Y,m)=(^2G_2(q),q^3+1)$.

Then $X_1< U$, so either $X_1=1$ and $N_a\le S_a$, or $X_1=Z(U)$, or $|U:X_1|=q$.
All these possibilities yield
$$N_Y(U)/X_1\le W^{\mathcal{B}}/X_1\simeq M^B/X_2=L\le (N_Y(U)/X_1).Z_f.Z_e.$$

In case (a) we have $SL_{\tilde{d}}(p^c)\le X_2$ and the group
$M^B/X_2$ is isomorphic to a section of the group $\Gamma L_{\tilde{d}}(p^c)/SL_{\tilde{d}}(p^c)\simeq Z_{(p^c-1)}.Z_c$.

In case (d) the group $M^B/X_2$ is isomorphic to a section of the group $\Gamma L_1(p^l)\simeq Z_{(p^l-1)}.Z_l$.

In any of these cases, $N_Y(U)/X_1\le Z_{(p^l-1)}.Z_l$ and, consequently, $UX_1/X_1$ can have at most two minimal normal subgroups of the same order.

If $Y={^2G}_2(q)$, then $N_Y(U)=U:D$, $D\simeq Z_{q-1}$, $\Phi(U)=U'$ is an elementary abelian subgroup of index $q$ in $U$, and $Z(U)$ is a subgroup of index $q$ in $U'$. Moreover, $D$ acts regularly on $Z(U)^{\#}$ and on $(U/U')^{\#}$, and has exactly two orbits of (odd) length $(q-1)/2$ on $(U'/Z(U))^{\#}$ (see \cite{Ward}).
But $UX_1/X_1\le Z_{(p^l-1)}.Z_l$, so $Z(U)\le X_1$. If $Z(U)= X_1$ or $U'\le X_1$, then $UX_1/X_1$ is elementary abelian of order $q^2$ or $q$, respectively, but this is also impossible due to the restriction $UX_1/X_1\le Z_{(p^l-1)}.Z_l$.
If $X_1$ contains an element of $U'-Z(U)$, then $U'\le X_1$.
This means that $Z(U)=X_1\cap U'$ and $X_1$ contains an element from $U-U'$. Therefore $U'X_1/X_1$ contains an elementary abelian group of order $q$, again a contradiction.

If $Y=U_3(q)$, then $\Phi(U)=U'=Z(U)$ is an elementary abelian subgroup of order $q$, $U/Z(U)$ is an elementary abelian group of order $q^2$, $N_Y(U)=U:D$, $D\simeq Z_{(q^2-1)/(q+1,3)}$ and $D$ acts fixed-point-freely on $(U/Z(U))^{\#}$.
But $UX_1/X_1\le Z_{(p^l-1)}.Z_l$, so $Z(U)\le X_1$. Moreover, it is easy to see that $Z(U)=U\cap X_1$.
Now if $X_1>U$, then $X_1/Z(U)$ centralizes $U/Z(U)$ and hence $X_1/Z(U)\le U/Z(U)$, a contradiction.
Therefore $Z(U)= X_1$, $UX_1/X_1$ is elementary abelian of order $q^2$ and has a cyclic complement in $N_Y(U)/X_1$ of order $(q^2-1)/(q+1,3)$, which contradicts the fact that $UX_1/X_1\le Z_{(p^l-1)}.Z_l$.

\end{itemize}

This completes the proof.
\end{proof}

\begin{theorem}\label{ImprimitiveCase:Thm3} If the groups $H$ and $G^{\mathcal{B}}$ are both affine, then
$t$ is even, $t-1=p_1^{s_1}=p_2^{s_2}-2$ for some   primes $p_1$ and $p_2$, and $\{|B|,|\mathcal{B}|\}=\{(t-1)^2,(t+1)^2\}$.
\end{theorem}
\begin{proof}

Suppose that $H$ and $G^{\mathcal{B}}$ are both affine. Then $n=(t^2-1)^2=p_1^{s_1}p_2^{s_2}$ for some distinct primes $p_1$ and $p_2$.

Suppose that $t$ is odd. Then $(t-1,t+1)=2=p_2$ and so $(t+1)^2=2^{2l}$ (by the assumption, $t-1$ cannot be a power of $2$). Hence $t-1=2p_1^{e_1}=2^l-2$ and $p_1^{e_1}=2^{l-1}-1$ by Lemma \ref{Prel:LemmaZsigmCor} is a Mersenne prime, $s_1=2e_1=2$, $s_2=2l+2$ and the number $l-1$ is prime.
Hence $r=(t-1)/2$.

 Then by Lemma \ref{Prel:LemmaNT} $x:=(m-1,k_1)=(2^{2(l+1)}-1,(2^{(l-1)}-1)^2-1)=(2^{2(l+1)}-1,2^{(l-2)}-1))=2^{(2(l+1),l-2)}-1=2^{(6,l-2)}-1<r-1$.
  As $\lambda=3t^3-7t+2> 4(t+1)^2$, by Lemma \ref{ImprimitiveCase:Lem1} we get that $(m-1)(k_1+1)/x$ divides $\lambda-\lambda_1$.
Hence, for some integer $y\ge 1$, we have
$$\lambda-\frac{y}{x}((t^2-1)^2-(k_1+1))=\lambda_1\ge 0.$$
This implies
$$2(3t^3-7t+2)\ge y\frac{r}{x}(t-1)(4t^2+8t+3)$$
if $k_1+1=r^2$, or
$$3t^3-7t+2\ge 4(t+1)^2y(r+1)\frac{r-1}{x}$$
if $m=r^2$.
But both possibilities give $y=1$ and $r/x<3/2$. Contradiction.

%To be continued...
\end{proof}

Finally, Theorem \ref{MainThm:ImprimitiveCase} follows  from Theorems \ref{ImprimitiveCase:Thm1}, \ref{ImprimitiveCase:Thm2} and \ref{ImprimitiveCase:Thm3}.

\section{Acknowledgements}

The work was carried out within the framework of the State Assignment of the Institute of Mathematics and Mechanics of the Ural Branch of the Russian Academy of Sciences, project FUMF-2022-0003 (for example, Section \ref{ImprimitiveCase})
and with the financial support of the Ministry of Science and Higher Education of the Russian Federation, project ``Ural Mathematical Center'', agreement No. 075-02-2024-1428 (for example, Section \ref{PrimitiveCase}).


\begin{thebibliography}{999}


\bibitem{Asch}
 M. Aschbacher,   Finite Group Theory, Second edition, Cambridge University Press,  2000.
 \bibitem{BCN} A.E. Brouwer, A.M. Cohen, A. Neumaier,   Distance-regular graphs,
Berlin etc: Springer-Verlag, 1989.
 \bibitem{Berggren}  J.L. Berggren, An algebraic characterization of finite symmetric tournaments. Bull. Austral. Math. Soc., 6 (1972) 53--59.
\bibitem{BM} A.E.~Brouwer, H.~Van~Maldeghem,   Strongly regular graphs, Cambridge University Press, 2022.
\bibitem{ChenPonomarenkoLNCC}  C.~Chen, I.~Ponomarenko, Lecture Notes on Coherent Configurations,   Oct 2023.
\bibitem{Atlas}  J.~Conway, R.~Curtis, S.~Norton, R.~Parker, R.~Wilson,  Atlas of finite groups,
Oxford: Clarendon Press, 1985.

\bibitem{CGSZ} G.~Coutinho, C.~Godsil, M.~Shirazi, H.~Zhan,   Equiangular lines and covers of the complete graph, Lin. Alg. Appl., {488} (2016), 264--283.
\bibitem{DGLPP}  A. Devillers, M. Giudici, C.H. Li, G. Pearce, Ch.E. Praeger,   On imprimitive rank 3 permutation groups, J. London Math. Soc. {84} (2011), 649--669.

\bibitem{DM}
J.D.~Dixon, B.~Mortimer,   Permutation Groups, Springer, New York, 1996.
\bibitem{FJMPW} M.~Fickus, J.~Jasper,  D.G.~Mixon, J.D.~Peterson, C.E.~Watson.  Polyphase equiangular tight frames and abelian generalized quadrangles, Appl. Comput. Harmon. Anal. 2019. 47:3 (2019) 628--661.
\bibitem{gap}
The~GAP~Group, GAP Groups, Algorithms, and Programming, Version 4.7.8, 2015.
\bibitem{GodsilHensel} C.D.~Godsil,  A.D.~Hensel,   Distance regular covers of the complete graph, J. Comb. Theory Ser. B., {56} (1992), 205--238.
\bibitem{GLP} C.D.~Godsil, R.A.~Liebler, C.E.~Praeger, Antipodal distance transitive covers
of complete graphs'', Europ. J. Comb., {19}:4 (1998), 455--478.


\bibitem{HuppertBlackburnPartII}
 B.~Huppert, N.~Blackburn.   Finite groups II,  New York: Springer-Verlag, 1982.
%
\bibitem{Ljunggren} W. Ljunggren, Noen setninger om ubestemte likninger av formen $\frac{x^n-1}{x-1}=y^q$, Norsk. Mat. Tidsskrift, 25 (1943), 17 -- 20.
 \bibitem{Mazurov1993}  V.D. Mazurov,  Minimal permutation representations of finite simple classical groups. Special linear, symplectic, and unitary groups,  {Algebra and Logic}, 32:3 (1993), 142--153.
  \bibitem{Tsi22c} L.Yu. Tsiovkina,  Covers of  complete graphs and related association schemes, J. Comb. Theory Ser. A., {191}:C (2022),   105646.

%  \bibitem{Tsi21}  L.Yu. Tsiovkina, On a class of vertex-transitive distance-regular covers of complete graphs. Sib. Elektron. Mat. Izv.,   8:2 (2021) 758--781. (in Russian)
%  \bibitem{Tsi22a} L.Yu. Tsiovkina, On a class of vertex-transitive distance-regular covers of complete graphs II. Sib. Electron. Mat. Izv.,  19:1 (2022) 348--359. (in Russian)
%  \bibitem{Tsi22b} Tsiovkina L. Yu. On some vertex-transitive distance-regular antipodal covers of complete graphs. Ural
% Math. J.,   8:2 (2022) 162--176.

\bibitem{Tsi24}  L.Yu. Tsiovkina. On $G$-vertex-transitive covers of complete graphs having at most two $G$-orbits on the arc set, Ural Math. J., {10}:1 (2024),  147--158.

 \bibitem{Vasil'ev1996} A.V. Vasil'ev  Minimal permutation representations of finite simple exceptional groups of types $G2$ and $F4$ . Algebr Logic 35, 371--383 (1996).

\bibitem{Ward} H.N. Ward, On Ree's series of simple groups, Trans. Am. Math. Soc., 121:1 (1966) 62--89.
\bibitem{Wilson} R.A.  Wilson, The finite simple groups, Grad. Texts in Math., vol. 251,  Springer-Verlag, London. 2009.
\bibitem{Zavarnitsine2004} A.V. Zavarnitsine, Recognition of the simple groups $L_3(q)$ by element orders, J. Group
 Theory, 7:1 (2004) 81--97.
\bibitem{Zsigmondy} K.  Zsigmondy, Zur Theorie der Potenzreste, Monatsh. Math. Phys.  3 (1892) 265--284.
\end{thebibliography}
 \end{document}